\documentclass[10pt]{amsart}

\usepackage{amsmath,amssymb,graphicx,bbm}
\usepackage{amsthm,verbatim}
\usepackage{mathrsfs}
\usepackage{xcolor}
\usepackage{array,multirow,booktabs} % for fancier tables
\usepackage{algorithm,algorithmic}

\usepackage{etoolbox}

\newtoggle{arxiv}
\togglefalse{arxiv}
\pdfoutput=1
\iftoggle{arxiv}{
  \pdfoutput=1
}{
}

\usepackage[footnotesize,bf]{caption}
\usepackage[left=1.25in,right=1.25in,top=1in]{geometry}

\usepackage{acncommands}

\newcommand*\Bell{\ensuremath{\boldsymbol\ell}}

\def\bc{\mathbf{c}}

\def\be{\mathbf{e}}
\def\boldf{\mathbf{f}}

\def\bi{\mathbf{i}}
\def\bj{\mathbf{j}}
\def\bk{\mathbf{k}}
\def\bl{\mathbf{\Bell}}

\def\bz{\mathbf{z}}

\def\bV{\mathbf{V}}

\def\bZ{\mathbf{Z}}

\def\cA{\mathcal{A}}

\def\cI{\mathcal{I}}
\def\cJ{\mathcal{J}}
\def\cK{\mathcal{K}}
\def\cL{\mathcal{L}}

\def\param{z}
\def\bxi{\boldsymbol{\param}}
\def\bXi{\boldsymbol{\uppercase{\param}}}

\newcommand{\supp}[1]{\mathrm{supp}(#1)}
\graphicspath{{figures/}}
\def\qoih{f}
\def\qoinh{f_n}
\def\rv{z}         % random variable
\def\brv{\bz}        % vector random variable
\newcommand{\intp}[2]{I_{#1}[#2]}
\newcommand{\dom}{I_{\bz}}
\def\bzero{\boldsymbol{0}}

\newcommand{\rev}[1]{#1}

\title[Leja sparse grid high-dimensional interpolation]{Adaptive Leja sparse grid constructions for stochastic collocation and high-dimensional approximation}
\author{Akil Narayan \and John D. Jakeman}

\begin{document}

\begin{abstract}
We propose an adaptive sparse grid stochastic collocation approach based upon Leja interpolation sequences for approximation of parameterized functions with high-dimensional parameters.  Leja sequences are arbitrarily granular (any number of nodes may be added to a current sequence, producing a new sequence) and thus are a good choice for the univariate composite rule used to construct adaptive sparse grids in high dimensions. When undertaking stochastic collocation one is often interested in constructing weighted approximation where the weights are determined by the probability densities of the random variables.  This paper establishes that a certain weighted formulation of one-dimensional Leja sequences produces a sequence of nodes whose empirical distribution converges to the corresponding limiting distribution of the Gauss quadrature nodes associated with the weight function.  This property is true even for unbounded domains. We apply the Leja-sparse grid approach to several high-dimensional and problems and demonstrate that Leja sequences are often superior to more standard sparse grid constructions (e.g. Clenshaw-Curtis), at least for interpolatory metrics.
\end{abstract}
\maketitle

\section{Introduction}
\iftoggle{arxiv}{
  Stochastic collocation (SC) has become a standard tool for non-intrusively quantifying uncertainty in simulation models that are subject to a degree of uncertainty or randomness.  Sources of uncertainty can include physical stochastic processes, parametric uncertainty or model form uncertainty~\cite{xiu_high-order_2005, nobile_sparse_2008,ng_multifidelity_2012}. If the sources of uncertainty can be parameterized by a set of random variables, the approaches used for uncertainty quantification (UQ) frequently reduce to computational methods that describe the behavior of a model with respect to those random variables. 

In this paper we will focus on the use of stochastic collocation methods that utilize sparse grid interpolants~\cite{smolyak_quadrature_1963,zenger_1991,gerstner_numerical_1998,bungartz04} to approximate the dependence of model output on unknown model parameters. Sparse grid stochastic collocation involves constructing an ensemble of random variables realizations, solving a deterministic physical system for each corresponding realization on the sparse grid, and using the resulting model output to build an approximation of the model response to the uncertain parameters. Once constructed, the sparse grid can be evaluated inexpensively to predict the variability of the physical model with respect to the random parameters.

Obtaining the ensemble of model solutions is usually the most expensive part of the collocation procedure since the model under consideration is often very complicated 
(e.g., complex geometry, multiscale features, stiff time-stepping, etc.). This dominance of model execution time on the computational expense of UQ motivates the need to build approximations of the model response that require as few model runs as possible. Sparse grids attempt to minimize the number or parameter realizations by generating ensembles that are geometrically sparse in the state space of the random variable.  Most efficient sparse grid constructions build up the ensemble adaptively by concentrating model evaluations in dimensions of the paramter space where the sparse grid approximation is poor~\cite{gerstner_2003,hegland_2003,ganapathysubramanian_2007}.

One of the crucial considerations when building a sparse grid is the choice of ensemble for the random variable.  The sparse grid is constructed via the union of judicious tensorizations of one-dimensional ensemble grids, and so the identification of these univariate grids is of paramount importance.  Common univariate choices are Clenshaw-Curtis or Chebyshev nodes and Gauss-quadrature-type nodes.

The desired characteristics in choosing a composite univariate rule for input into the sparse grid include: 
\begin{itemize}
  \item efficiency -- high interpolation and/or quadrature accuracy with low cardinality sets
  \item robustness -- consistent and increasing accuracy when the grid is refined
  \item monotonicity -- fine-level grids are supersets of coarse-level grids
  \item granularity -- the number of nodes needed for refinement of a grid is as small as possible
\end{itemize}
Efficiency and robustness are desired when using approximation grids in any context. The monotonicity property is motivated mainly by the high cost of solving deterministic physical simulations and the sparse grid algorithm: if monotonicity holds, then the tensorized sparse grid construction has many fewer total nodes. Granularity becomes important when several levels of refinement are necessary: it is much better to have the ability to add a small number nodes for each refinement step than to be required to, e.g., double the number of nodes.

In this paper we propose use of univariate Leja sequences for use in the sparse grid algorithm for interpolatory high-order and high-dimensional approximation. Leja rules are a sequence of interpolation/quadrature grids in one dimension that are strongly monotonic and granular: coarse grids are always strict subsets of fine grids, and refinement proceeds by adding a single node at a time. Leja sequences are very accurate as we show later, but they are not as efficient as Gauss-type rules in some cases (e.g., quadrature). Therefore, we argue that Leja sequences are a grid choice that serves as a good compromise of the above desired characteristics, in contrast to, e.g. a Gauss quadrature grid that is strong in efficiency and granularity, but very weak in monotonicity. 

For many weight functions of interest (i.e. the random variable probability density function), we show that (our definition of) weighted Leja sequences produce a nodal sequence whose asymptotic distribution coincides with the asympototic distribution of Gauss quadrature nodes associated to the polynomial family orthogonal under the weight function. This result is known for the uniform-weight case; to our knowledge the weighted versions are novel. Additionally, we show that a contracted version of weighted Leja sequences are \textit{asymptotically weighted Fekete}, meaning that their Vandermonde determinant grows comparably to the largest possible value (Fekete). That these results hold for unbounded domains is significant as it suggests that Leja sequences will be accurate for nested interpolatory approximation when the random variable state space is infinite.

In Section \ref{sec:setup} we setup the problem and introduce notation and terminology. In Section \ref{sec:leja} we introduce Leja sequences and formally present the above-mentioned properties. Section \ref{sec:sparse-grids} develops the methodology for adaptive sparse grids. Finally, Section \ref{sec:results} shows that the Leja sparse grid algorithm produces results comparable to well-established sparse grid approximation methods, and in many cases is superior. The proof of our main result in Section \ref{sec:leja} concerning the distribution of Leja sequence nodes is relatively involved, employing results from weighted potential theory; for this reason we leave this until the end in Section \ref{sec:proofs}, serving somewhat as an appendix.

The first half of this paper (Section \ref{sec:leja}) proves certain results about one-dimensional weighted Leja sequences and compares them to other one-dimensional rules. The second half (Sections \ref{sec:sparse-grids} and \ref{sec:results}) uses these weighted Leja sequences for building adaptive multivariate sparse grids.

\section{Setup}\label{sec:setup}
A model problem in the UQ community that serves as a motivating example is a parameterized elliptic equation where the parameters $\bZ$ are random variables:
\begin{align}\label{eq:model-equation}
-\frac{d}{dx}\left[a(x,\bZ)\frac{du}{dx}(x,\bZ)\right]=f(x,\bZ),\quad 
(x,\bZ)\in(0,1)\times I_\bz
\end{align}
This model describes the steady-state temperature distribution $u$ in a one-dimensional domain where the domain has diffusivity coefficient $a$ and experiences an external heat source defined by $f$. Here $x$ is a spatial variable taking values in a one-dimensional interval domain. The variable $\bZ$ is a random vector with density function $\omega(\bZ)$ on domain $I_{\bz}$ corresponding to a probability measure $P$ on a complete probability space. The diffusion coefficient $a$ is a model parameter that varies spatially, but is also influenced by uncertainty. Uncertainty in these kinds of models may arise from, e.g., imprecise knowledge of material parameters or external forcing. Under the assumption that the equation is well-posed almost surely, the solution $u(x, \bZ)$ is random, and essentially depends on $\left(1 + \dim \bZ\right)$ variables. One goal in the UQ community is efficient and accurate prediction of $u(x,\bZ)$ or some quantities of interest that depend on $u$ (e.g. the temperature variance as a function of space $x$). 

One popular technique is the generalized Polynomial Chaos (gPC) approach: we assume that the variation of $u$ with respect to the random parameter $\bZ$ can be described accurately by a finite-degree polynomial:
\begin{align*}
  u(x,\bZ) \simeq u_N(x, \bZ) = \sum_{n=1}^N \widehat{u}_n(x) \phi_n(\bZ),
\end{align*}
where $\phi_n$ are polynomials that satisfy an orthogonality condition
\begin{align*}
  \E [ \phi_n(\bZ) \phi_m(\bZ) ] = \int_{I_z} \phi_n(\bz) \phi_m(\bz) \omega(\bz) \dx{\bz} = \delta_{n,m},
\end{align*}
and the $\widehat{u}_n$ are coefficient functions that depend only on the spatial variable $x$. The determination of the functions $\widehat{u}_n$ is the challenge, and one straightforward procedure is to use a probabilistic sampling strategy to compute these coefficients: let $\bz_m$ for $1 \leq m \leq M$ be given samples of the variable $\bZ$. For each $\bz_m$, equation \eqref{eq:model-equation} is a deterministic differential equation, and any computational simulation or experimental setup may be used to obtain $u(x, \bz_m)$. (Frequently, this solution is a finite-dimensional quantity rather than a function of a continuum variable $x$, but this distinction does not affect the main focus of our discussion.) Once these realizations of $u(x,\bZ)$ are collected, then we attempt to find $\widehat{u}_n(x)$ such that 
\begin{align*}
  \sum_{n=1}^N \phi_n(\bz_m) \widehat{u}_n(x) &\approx u(x, \bz_m), & m &= 1, \ldots, M.
\end{align*}
If one considers $u$ as a scalar, then this is a linear algebra problem, seeking a vector $\bf{\widehat{u}}$ that solves:
\begin{align*}
  \bf{A} \bf{\widehat{u}} \approx \bf{u}.
\end{align*}
This problem may be solved by defining $\approx$ in any convenient fashion: interpolation, least-squares regression, quadrature, or compressive sampling. See, e.g. \cite{xiu_fast_2009,xiu_numerical_2010,cohen_stability_2013,yan_stochastic_2012}. Usually the particular choice made is dependent on the relationship between $M$ and $N$ (determined by the computational cost of computing each solution realization) and dependent on some \textit{a priori} understanding of the accuracy for the choice.

One major difficulty with this approach is the selection of $\bz_m$ when $\dim \bZ$ is large. While the spatial variable $x$ is usually restricted to have dimension less than or equal to 3, it is not uncommon to have 100 or more parameters as the components of $\bZ$. High-dimensional approximation has been a persistent bottleneck in modern scientific computing. Methods that work very well for a small number of dimensions are rendered ineffective or impossible to implement in a large number of dimensions, owing to the curse of dimensionality: exponential dependence of functional complexity with respect to the parametric dimension (when the functional smoothness is fixed). Tensor product constructions and space-filling designs adopt this complexity with respect to dimension. 

Attempting to circumvent the computationally onerous space-filling property is the main reason to consider alternatives such as sparse grids. The sparse grid construction still employs tensor product grids, but it does so in a way that attempts to control the cardinality of the mesh and delays the curse of dimensionality. Sparse grids are formed by (unions of) tensorized one-dimensional grids, and therefore an educated selection of the composite one-dimensional rules is of great importance. 

In this paper we consider high-order interpolatory approximation using a sparse grid, and we employ weighted Leja grids as the one-dimensional composite rules. Weighted Leja grids are nested grids (they are a sequence), and we prove that the nodes distribute identically to the one-dimensional Gauss quadrature rules. Thus, Weighted Leja sequences distribute nodes in a way that emulates the Gauss quadrature rule, and have the advantage of being nested. 

Gauss-Kronrod \cite{laurie_calculation_1997,calvetti_computation_2000,gautschi_orthogonal_2004} and Gauss-Patterson \cite{patterson_optimum_1968} rules are likewise nested interpolatory schemes, but their computation is usually restricted to special weight functions because computation of the nested rules is relatively difficult. In contrast, weighted Leja sequences are simple and very easy to compute (exactly and approximately) even for exotic weight functions.

\rev{Our approach with weighted Leja sequences considers polynomial approximation on unbounded domains, but there are alternative high-order approaches. As described in \cite{boyd_chebyshev_2001}, there are three popular approaches to high-order approximation on unbounded domains: (1) expansions on infinite domains using polynomial or non-polynomial complete basis sets \cite{stenger_numerical_1981, narayan_generalization_2011} (2) domain trunction, where the full domain $I_z$ is replaced by a compact subset \cite{boyd_chebyshev_1988,boyd_optimization_1982}, and (3) mapping techniques \cite{weideman_spectral_1990,grosch_numerical_1977,gottlieb_numerical_1987} where standard methods on compact intervals are ``transplanted" onto an infinite interval via a domain mapping. Each of these methods can perform accurate approximation on unbounded domains depending on the application. Our approach is most closely related to (1), but in principle one may use Leja sequences for any of the above methods. However, this application is outside the scope of this paper.}

For the remainder of this paper, we replace the uppercase variable $\bZ$ (traditionally denoting a random quantity) with its lowercase counterpart $\bz$: stochasticity does not affect our approach so in principle we may treat the random parameter as a deterministic parameter $\bz$ with corresponding weight function $\omega$.

%In Section \ref{sec:leja} we introduce weighted Leja sequences, which are the one-dimensional building blocks for our method. This section also enumerates the desirable properties that Leja sequences possess in one dimension. Section \ref{sec:sparse-grids} introduces the adaptive, high-order Smolyak sparse grid approach for multivariate SC, and discusses the use of weighted Leja rules as the one-dimensional composite rules. Finally, in Section \ref{sec:results} we show several multivariate results showing that Smolyak-Leja constructions are competitive with, and in many cases superior to, the more widely used Smolyak-Clenshaw-Curtis method.

}{
  
}

\section{Weighted Leja points}\label{sec:leja}
\iftoggle{arxiv}{
  In this section we present and establish important properties of univariate Leja sequences. Consider approximation in the scalar variable $z$ over the domain $I$ in the presence of a weight function $w$. In the context of this paper, $z$ represents one component of the vector-valued parameter $\bz$, $I$ is the one-dimensional restriction of $I_{\bz}$ to the appropriate dimension, and $w(z)$ is the marginal density computed from the joint density $\omega(\bz)$.

A Leja sequence (unweighted) is classically defined \cite{edrei_sur_1940,leja_sur_1957} as a sequence of points $z_n \in I = [-1,1] \subset \R$ for $n = 1, 2, \ldots,$ such that 
\begin{align}\label{eq:unweighted-leja-objective}
  z_{N+1} = \argmax_{z \in [-1,1]} \prod_{n=0}^N \left| z - z_n \right|,
\end{align}
where the starting point of the sequence $z_0$ is arbitrarily chosen in $[-1,1]$. We note that it is only sensible to define the above Leja sequences on bounded domains. We list below the properties of one-dimensional (unweighted) Leja sequences:
\begin{itemize}
  \item Leja sequences are not unique. The initial point $z_0$ may be arbitrarily chosen and the objective function being maximized need not have a unique maximizer.
  \item The Leja construction provides an interpolation \textit{sequence}. Therefore if $\{z_1, \ldots, z_7\}$ are a Leja sequence, then we require only one more point $z_8$ to construct a higher-order interpolant. This addresses the granularity criterion for grids, and will be useful in minimizing the number of function evaluations necessary for approximation in high dimensions.
  \item Maximizing the objective function \eqref{eq:unweighted-leja-objective} is equivalent to a greedy determinant maximization (e.g., \cite{de_marchi_leja_2004}). With $z_1, \ldots, z_{N-1}$ specified, let $V_N(z)$ be the $N \times N$ interpolatory Vandermonde-like matrix for the space $\Pi_{N-1}$ at the collocation points $z_1, \ldots, z_{N-1}, z$. (The choice of basis for $\Pi_{N-1}$ does not affect maximization.) Then \eqref{eq:unweighted-leja-objective} is equivalent to
  \begin{align*}
    z_{N+1} = \argmax_{z \in [-1,1]} \left| \det V_N(z) \right|
  \end{align*}
  Thus one can view Leja sequences as a greedy $D$-optimal experimental design \cite{fedorov_theory_1972}.
  \item The Lebesgue constant for interpolation on a Leja sequence grows subexponentially \cite{taylor_lebesgue_2010}.
  \item Leja sequences are asymptotically Fekete \cite{bloom_polynomial_1992}. (This is implied by the subexponentially growing Lebesgue constant.) Fekete points are those whose Vandermonde determinant is as large as possible. (These are known to be Gauss-Lobatto nodes in one dimension \cite{fejer_bestimmung_1932}.) The asymptotically Fekete property essentially means that the Vandermonde determinant of a Leja sequence grows on par with that of true Fekete nodes.
  \item Any Leja sequence asymptotically distributes according to the Chebyshev (arcsine) measure. (This is implied by the asymptotically Fekete property \cite{bloom_polynomial_1992}.)
  \item In practice, optimization of \eqref{eq:unweighted-leja-objective} over a discrete candidate set can be accomplished in computationally efficient ways \cite{baglama_fast_1998} and with standard numerical linear algebra routines: E.g., the row permutation information from a row-pivoted $L U$ matrix decomposition gives the Leja sequence order, e.g. \cite{bos_computing_2010}. 
  %\item Leja sequences are excellent interpolation nodes, with Lebesgue constant that grows logarithmically.
\end{itemize}
\begin{figure}
\begin{center}
  \iftoggle{arxiv}{
    \resizebox{\textwidth}{!}{\includegraphics{empirical-distribution}}
  }{
    \resizebox{\textwidth}{!}{\includegraphics{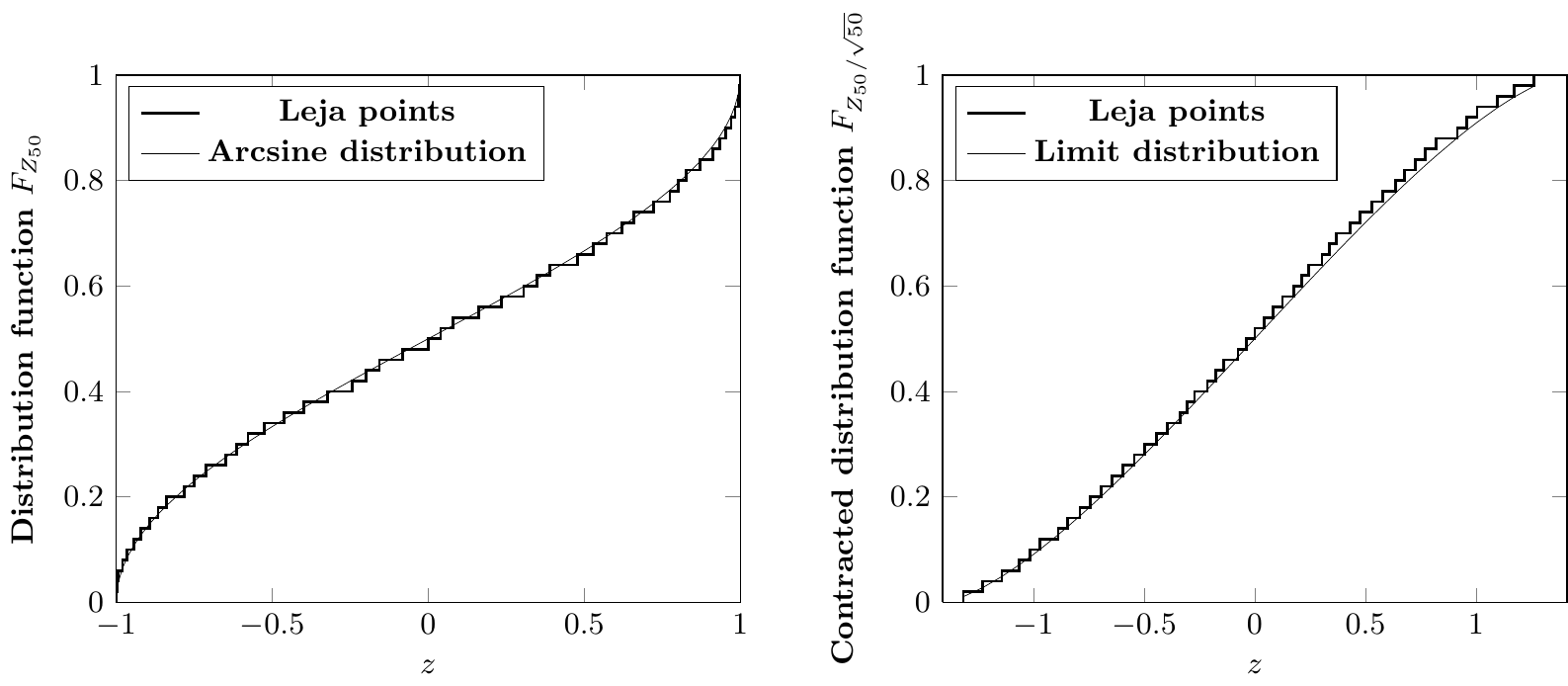}}
  }
\end{center}
\caption{Theoretical asymptotic distribution of Gauss quadrature nodes for the family of polynomials orthogonal under weight $\omega$, and empirical distributions for the first 50 associated Leja points (denoted $Z_{50}$ in both cases). Left: uniform weight $w \equiv 1$ with limiting (arcsine) distribution. Right: Gaussian density weight $w(z) = \exp(-z^2)$ with a 50-point Leja sequence, contracted to the origin by a factor of $\sqrt{50}$, along with the contracted limiting distribution of the Gauss quadrature nodes. (See tables \ref{tab:summary-table} and \ref{tab:hermite-table})}\label{fig:univariate-leja-distribution}
\end{figure}

The Leja sequences introduced above are quite useful for unweighted polynomial interpolation. However, in the UQ context we are usually interested in interpolation involving a weight function (here, the marginal density of the random variables). Therefore, we are also interested in \textit{weighted} Leja sequences. Let $w(z)$ be a continuous and positive weight function on a univariate domain $I$, and let 
\begin{align}\label{eq:square-root-weight}
v(z) = \sqrt{w(z)}
\end{align}
be its square root. If $I$ is unbounded, we assume that polynomials are dense in the space of continuous functions, measured in the $v$-weighted supremum norm. For example, if $v(z) \propto \exp(- |z|^\alpha)$, then we require $\alpha \geq 1$, e.g. \cite{lubinsky_survey_2007}. This density assumption is necessary in our context: we cannot hope to form an accurate polynomial approximation without polynomial density.

A weighted Leja sequence can be constructed via the optimization
\begin{subequations}\label{eq:weighted-leja-objective}
\begin{align}\label{eq:maximizer}
  z_{N+1} = \argmax_{z \in I} \sqrt{w(z)} \prod_{n=0}^N \left| z - z_n \right| = \argmax_{z \in I} v(z) \prod_{n=0}^N \left| z - z_n \right|.
\end{align}
In the case where multiple choices of $z$ maximize the objective, for concreteness we choose the one with smallest magnitude, i.e.,
\begin{align}\label{eq:maximizer-choice}
  z_{N+1} &= \argmin_z \left\{|z|\,\, |\,\, z \in f^{-1}\left( W \right) \right\},
  & W = \max_{z \in I} f(z) \triangleq \max_{z \in I} v(z) \prod_{n=0}^N \left| z - z_n \right|.
\end{align}
\end{subequations}
The sequence of points $z_n$ produced by the above iteration is the central study of this paper, and we refer to this sequence as a weighted Leja sequence, or a $w$-weighted Leja sequence. This formulation still leaves a benign ambiguity when multiple maximizers differ only in sign. \rev{In this paper, Leja sequence optimization is one-dimensional, so in all that follows we optimize via \eqref{eq:weighted-leja-objective} exactly (up to machine accuracy).}

In general there is no standard choice of how to incorporate the weight function into a Leja objective; we have chosen $v = \sqrt{w}$, but alternatives have been proposed \cite{de_marchi_leja_2004, saff_logarithmic_1997}. Our choice above is motivated by the fact that under this formulation the sequence of points produced has the same asymptotic distribution as $w$-Gauss quadrature nodes. We illustrate this property now: In Figure \ref{fig:univariate-leja-distribution} we show that the distribution of a univariate 50-point weighted Leja sequence seems to converge to the distribution of the Gauss points associated with the family of polynomial orthogonal under the weight $w$. Therefore, the objective \eqref{eq:weighted-leja-objective} produces points that are `approximately' Gauss nodes, with the additional benefit of being nested. This distribution property alone does not guarantee that Leja sequences are useful, but we provide several examples in this paper that suggest that Leja points have utility.

We emphasize that unlike the unweighted case \eqref{eq:unweighted-leja-objective}, weighted Leja sequences are constructible on unbounded domains given our assumptions. In Figure \ref{fig:whack-a-mole} we show a graphical illustration of the iterative Leja procedure that produces a Leja sequence.

\begin{figure}
\begin{center}
  \iftoggle{arxiv}{
    \resizebox{\textwidth}{!}{\includegraphics{whack-a-mole}}
  }{
    \resizebox{\textwidth}{!}{\includegraphics{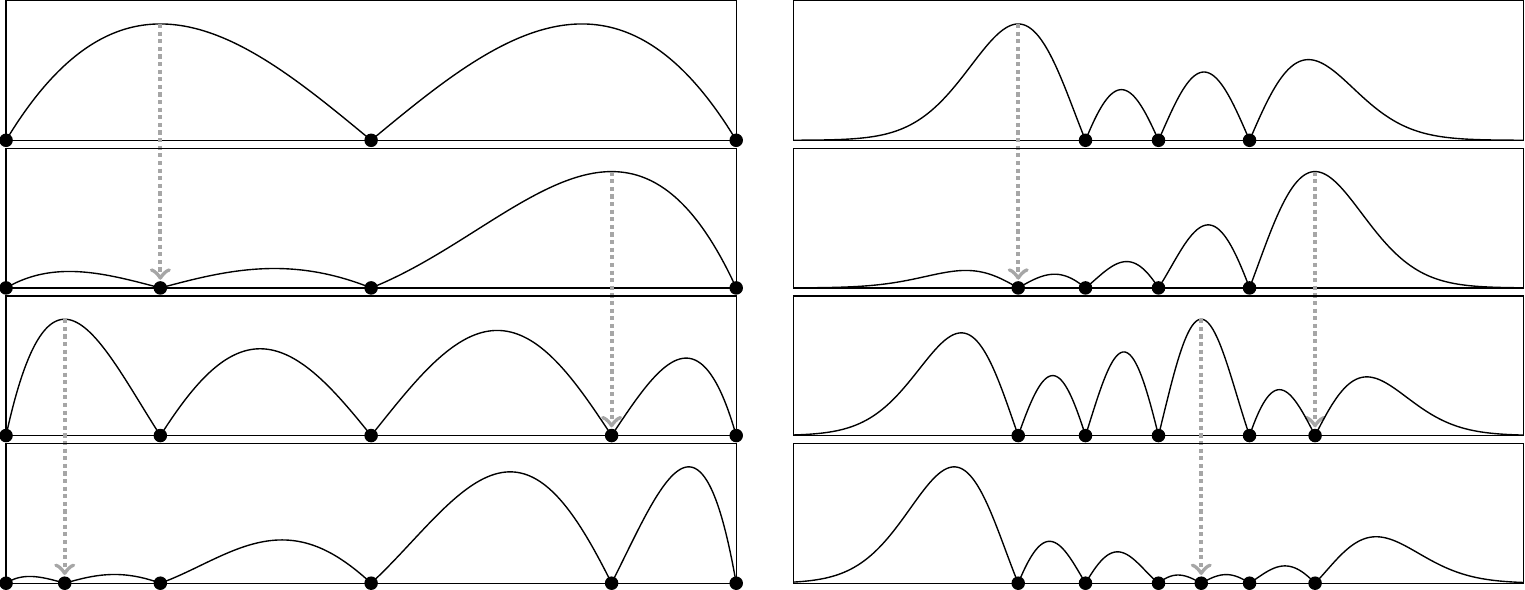}}
  }
\end{center}
\caption{Sequential addition of nodes (top to bottom) via iteration on the Leja objective \eqref{eq:weighted-leja-objective}. Left: A Leja sequence on $[-1,1]$ with $v = w \equiv 1$ (i.e., unweighted). Right: A weighted Leja sequence on $\R$ with $v^2 = w(z) = \exp(-z^2)$.}\label{fig:whack-a-mole}
\end{figure}

\subsection{Limits of weighted Leja sequences}
A significant concern for the unbounded case with the weighted Leja sequences we have introduced above is that they `do the right thing', i.e. that they produce a set of nodes that is desirable from the approximation theory point of view.\footnote{Alternative propositions for weighted Leja points \cite{saff_logarithmic_1997} construct points on a compact subset of the domain. These points are asymptotically optimal if we are interested in approximation with a basis $v^n p_n$.But because we are interested in approximation on an unbounded domain with just $v p_n$, we require samples to be produced on the entire domain.}

The purpose of this subsection is to show that the weighted Leja sequences we have proposed \eqref{eq:weighted-leja-objective} produce points whose asympototic distribution is identical to the Gauss quadrature nodes associated with the corresponding $w$-orthogonal polynomial family. To be precise, given a classical weight function $w$, it is known that there is another weight function $\widetilde{v}$ such that the empirical distribution of the $w$-Gauss quadrature nodes converges to the $\widetilde{v}$-weighted potential equilibrium measure; roughly speaking, $\widetilde{v} \sim \sqrt{w}$, modulo multiplicative polynomial factors. We show that for this same class of weight functions, the empirical measure of the Leja sequence converges to the same equilibrium measure. On unbounded domains with exponential weights, $\widetilde{v}$ is the square root of $w$. On bounded domains with Jacobi-type weights, $\widetilde{v}$ is the uniform weight.

\rev{A comprehensive discussion of potential theory and equilibrium measures can be found in \cite{saff_logarithmic_1997}; here and later in Section \ref{sec:proofs}, which contains proofs, we give a brief account of this topic.} Let $\widetilde{v}$ be a weight function on $I$ that we will precisely relate to $w$ later; define $Q \triangleq -\log \widetilde{v}$. If $\widetilde{v}$ is admissible\footnote{$v$ must be admissible in the sense of potential theory: (1) it is a non-negative upper semicontinuous function (2) $I \cap \widetilde{v}^{-1}((0, \infty])$ is nonpolar in the sense of potential theory (a polar is `negligible' and has Lebesgue measure 0), and (3) if $I$ is unbounded, $|z| \widetilde{v}(z) \rightarrow 0$ as $|z|\rightarrow \infty$} on $I$, then there is a unique probability measure $\mu_{\widetilde{v}}$ that minimizes a weighted logarithmic energy. This measure $\mu_{\widetilde{v}}$ is the logarithmic potential equilibrium measure of the domain $I$ in the presence of the external field (i.e. weight) $Q$. (See Section \ref{sec:proofs} for a more detailed discussion.) When $\widetilde{v}$ is the uniform measure on a compact interval $I$, then $\mu_{\widetilde{v}}$ is the arcsine, or Chebyshev measure. %In this uniform case then $Q = 0$ and we will write $\mu_{\widetilde{v}} = \mu_{I}$.

For each $N \in \N$, let $\xi_{n,N}$ with $n = 1, \ldots, N$ denote the $N$ zeros of the degree-$N$ polynomial orthogonal under $w(z)$. I.e., $\xi_{n,N}$ are the $N$-point $w$-Gauss quadrature nodes. Let $z_n$ denote any sequence of weighted Leja nodes given by \eqref{eq:weighted-leja-objective}. The $\xi_{n,N}$ are a triangular array ($n \leq N$) while the $z_n$ are a sequence. We introduce a contraction factor $k_n$ defined in the following theorem; this contraction factor is used to define the empirical (counting) measure for an $N$-point Gauss ($\xi_{n,N}$) and Leja ($z_n$) grid, respectively:
\begin{align*}
  \nu^G_N &= \frac{1}{N} \sum_{n=1}^N \delta_{\left(k_N \xi_{n,N}\right)}, &
  \nu^L_N &= \frac{1}{N} \sum_{n=1}^N \delta_{\left(k_N z_{n}\right)},
\end{align*}
where $\delta_{z}$ is the Dirac distribution centered at $z$. Our main result in this section shows that for most classical univariate weight functions of interest, $\nu^G_N$ and $\nu^L_N$ limit to the same measure.
\begin{theorem}\label{thm:main-leja-result}
  Let $w(z)$ be a weight function on $I$.
  \begin{enumerate}
    \item (Generalized Hermite) \newline 
      Let $w(z) = z^{2\mu} \exp(-|z|^\alpha)$ for any $\alpha \geq 1$, $\mu > -\frac{1}{2}$ on $I = \R$. Define $k_n = n^{-1/\alpha}$ and $\widetilde{v} = \exp\left(-\frac{1}{2} |z|^\alpha\right)$.
    \item (Laguerre) \newline 
      Let $w(z) = z^s \exp(- |z|)$ for any $s > -1$ on $I = [0, \infty)$. Define $k_n = n^{-1}$ and $\widetilde{v}(z) = \exp\left(-\frac{1}{2} z \right)$.
    \item (Jacobi) \newline 
      Let $w(z) = (1-z)^\alpha (1+z)^\beta$ for any $\alpha,\beta > -1$ on $I = [-1,1]$. Define $k_n \equiv 1$ and $\widetilde{v}(z) \equiv 1$.
  \end{enumerate}
  In all of the above cases, we have
  \begin{align}\label{eq:measure-equality}
    \lim_{n\rightarrow\infty} \nu^G_n = \mu_{\widetilde{v}} 
    = \lim_{n\rightarrow\infty} \nu^L_n,
  \end{align}
  where equality holds in the weak sense.
\end{theorem}
\begin{remark}
  The ``Hermite" result from the Theorem above that $\lim_n \nu_n^L \rightarrow \mu_{\widetilde{v}}$ also holds for $0 < \alpha < 1$. However, since polynomials are not dense for this weight function \cite{lubinsky_survey_2007}, it is unclear if one should use polynomial approximation in this case.
\end{remark}
The above theorem states that the Leja sequence produced by \eqref{eq:weighted-leja-objective} produces a nested mesh whose samples distribute precisely like $w$-Gaussian quadrature nodes. We emphasize that while this property is promising, it does not guarantee a good approximation: for example, one can generate a grid according to the acrsine measure on $I = [-1,1]$ whose Lebesgue constant does not grow subexponentially \cite{bloom_polynomial_1992}. However, an unweighted Leja sequence is known to have subexponentially growing Lebesgue constant (which is `good enough' in a sense for approximating very smooth functions). In practice, unweighted Leja sequence have logarithmically-growing Lebesgue constant. However, to our knowledge it is presently unknown if weighted Leja sequences have subexponentially growing weighted Lebesgue constant.

The portion of \eqref{eq:measure-equality} that relates the Gauss quadrature node distribution (zero distribution of orthogonal polynomials) to the measure $\mu_{\widetilde{v}}$ is well-known: \cite{mhaskar_extremal_1984,nevai_asymptotic_1979,rakhmanov_asymptotic_1984,ullman_orthogonal_1980,mhaskar_extremal_1983,saff_logarithmic_1997}. That the unweighted Leja node distribution converges to the arcsine measure is likewise well-known. Our novel contribution to result \eqref{eq:measure-equality} is for the limit for the weighted Leja formulation \eqref{eq:weighted-leja-objective}.

We give a summary of the weights $w$, contraction factors $k_n$, and some details about the asymptotic measures $\mu_{\widetilde{v}}$ in Table \ref{tab:summary-table}. The formulas for the exponential weight $w(z) = \exp(-|z|^\alpha)$ are not explicit for general $\alpha$, so we explicitly compute and collect the density and distribution expressions for a selection of values for $\alpha$ in Table \ref{tab:hermite-table}. Finally, the densities associated to $\mu_{\widetilde{v}}$ are plotted for these specials cases in Figure \ref{fig:asymptotic-density}.

\begin{table}
  \renewcommand{\tabcolsep}{0.4cm}
  \renewcommand{\arraystretch}{1.3}
  \begin{center}
  \resizebox{\textwidth}{!}{
  \begin{tabular}{@{}cccccc@{}}\toprule
    Class & Domain & Parameters & Weight & Contraction $k_n$ & Equilibrium weight \\ \midrule
  Hermite & $I = \R$ & $\alpha \in [1, \infty),\,\, \mu \in \left(-\frac{1}{2}, \infty\right)$ & $w(z) = z^{2\mu} \exp(-|z|^\alpha)$ & $k_n = n^{-1/\alpha}$ & $\widetilde{v}(z) = \exp\left(-\frac{1}{2} |z|^\alpha\right)$ \\
    Laguerre & $I = [0, \infty)$ & $s \in (-1, \infty)$ & $w(z) = z^s \exp( -z)$ & $k_n = n^{-1}$ & $\widetilde{v}(z) = \exp\left(- \frac{1}{2} z\right)$ \\
    Jacobi & $I = [-1, 1]$ & $\alpha, \beta \in (-1, \infty)$ & $w(z) = (1 - z)^\alpha (1 + z)^\beta$ & $k_n \equiv 1$ & $\widetilde{v} \equiv 1$ \\
  \bottomrule
  \end{tabular}
}
  \end{center}
  \renewcommand{\arraystretch}{1}
  \renewcommand{\tabcolsep}{12pt}
  \vskip 10pt
  \begin{center}
  \resizebox{\textwidth}{!}{
  \renewcommand{\tabcolsep}{0.4cm}
  \renewcommand{\arraystretch}{1.8}
  \begin{tabular}{@{}cccc@{}}\toprule
    Class & $\supp{\mu_{\widetilde{v}}} = [a,b]$ & Distribution $F(t) = \mu_{\widetilde{v}}\left[\,\left(-\infty, t\right]\,\right]$ & Density $f(t) = \dfdx{\mu_{\widetilde{v}}}{t}$ \\ \midrule
  Hermite & $-a = b = \left[ 2^{\alpha-1} B\left(\frac{\alpha}{2}, \frac{\alpha}{2}\right)\right]^{1/\alpha}$ & See Table \ref{tab:hermite-table} & $f(t) = \frac{\alpha}{\pi b^\alpha} \int_{|t|}^b  \frac{u^{\alpha-1}}{\sqrt{u^2 - t^2}} \dx{u}$ \\
  Laguerre & $a=0,\,\, b = 4$ & $F(t) = \frac{2}{\pi} \arcsin \left(\frac{\sqrt{t}}{2}\right) + \frac{1}{2\pi}\sqrt{t(4-t)}$ & $f(t) = \frac{1}{2 \pi} \sqrt{\frac{4-t}{t}}$ \\
    %Jacobi & $a = u,\,\, b = v$ & $F(t) = \frac{1}{2} + \frac{1}{\pi} \arcsin\left(\frac{2 t - u - v}{v-u}\right)$ & $f(t) = \frac{2}{\pi(v-u)} \frac{1}{\sqrt{(v-z)(z-u)}}$ \\
    Jacobi & $a = -1,\,\, b = 1$ & $F(t) = \frac{1}{2} + \frac{1}{\pi} \arcsin t$ & $f(t) = \frac{1}{\pi} \frac{1}{\sqrt{1-t^2}}$ \\
  \bottomrule
  \end{tabular}
  \renewcommand{\arraystretch}{1}
  \renewcommand{\tabcolsep}{12pt}
}
\end{center}
\caption{Asymptotic distributions and densities to which contracted weighted Leja sequences (and Gauss quadrature nodes) converge. Given a weighted Leja sequence $z_n$, the formulae to which the quantities in this table correspond is given by Theorem \ref{thm:main-leja-result}.}\label{tab:summary-table}
\end{table}

\begin{table}
  \begin{center}
\resizebox{\textwidth}{!}{
  \renewcommand{\tabcolsep}{0.4cm}
  \renewcommand{\arraystretch}{1.8}
  \begin{tabular}{@{}c|cccc@{}}\toprule
    & & & Distribution $F(t) = \mu_{\widetilde{v}}\left[\,\left(-\infty, t\right]\,\right]$ & Density $f(t) = \dfdx{\mu_{\widetilde{v}}}{t}$ \\
  & & & $t \in [-b(\alpha), b(\alpha)]$ & $t \in [-b(\alpha), b(\alpha)]$ \\\midrule
  %& $\alpha = 1$ & $b = \pi$ & $\frac{1}{2} + \frac{1}{\pi}\arcsin\left(\frac{t}{b}\right) + \frac{t}{\pi^2} \log\left[\frac{b + \sqrt{b^2 - t^2}}{|t|}\right]$ & $\frac{1}{\pi^2} \log\left[ \frac{b + \sqrt{b^2 - t^2}}{|t|} \right]$ \\
  & $\alpha = 1$ & $b = \pi$ & $\frac{1}{2} + \frac{1}{\pi}\arcsin\left(\frac{t}{b}\right) + t f(t)$ & $\frac{1}{\pi^2} \log\left[ \frac{b + \sqrt{b^2 - t^2}}{|t|} \right]$ \\
    %Hermite & $\alpha = 2$ & $b = \sqrt{2}$ & $F(t) = \frac{1}{2} + \frac{1}{\pi} \arcsin\left(\frac{t}{b}\right) + \frac{t\sqrt{b^2-t^2}}{2 \pi}$ & $f(t) = \frac{1}{\pi}\sqrt{b^2-t^2}$\\
    & $\alpha = 2$ & $b = \sqrt{2}$ & $\frac{1}{2} + \frac{1}{\pi} \arcsin\left(\frac{t}{b}\right) + \frac{t}{2} f(t)$ & $\frac{1}{\pi}\sqrt{b^2-t^2}$\\
            %& $\alpha = 3$ & $b = \left(\frac{\pi}{2}\right)^{1/3}$ & $\frac{1}{2} + \frac{1}{\pi} \arcsin\left(\frac{t}{b}\right) + \frac{t}{\pi^2}\left[b \sqrt{b^2 - t^2} + t^2 \log\left[\frac{b + \sqrt{b^2 - t^2}}{|t|}\right]\right]$ & $\frac{3}{\pi^2} \left[b \sqrt{b^2 - t^2} + t^2 \log \left[ \frac{b + \sqrt{b^2 - t^2}}{|t|} \right] \right]$ \\
     $w(z) = \exp(-|z|^\alpha)$ & $\alpha = 3$ & $b = \left(\frac{\pi}{2}\right)^{1/3}$ & $\frac{1}{2} + \frac{1}{\pi} \arcsin\left(\frac{t}{b}\right) + \frac{t}{3} f(t)$ & $\frac{3}{\pi^2} \left[b \sqrt{b^2 - t^2} + t^2 \log \left[ \frac{b + \sqrt{b^2 - t^2}}{|t|} \right] \right]$ \\
    %$w(z) = \exp(-|z|^\alpha)$ & $\alpha = 4$ & $b = \left(\frac{4}{3}\right)^{1/4}$ & $F(t) = \frac{1}{2} + \frac{1}{\pi} \arcsin\left(\frac{t}{b}\right) + \frac{t \sqrt{b^2-t^2}}{4\pi} \left(2 t^2 + b^2 \right)$ & $f(t) = \frac{1}{\pi}\sqrt{b^2-t^2}\left( 2 t^2 + b^2 \right)$ \\
    $b(\alpha) = \left[2^{\alpha-1} B\left(\frac{\alpha}{2},\frac{\alpha}{2}\right)\right]^{1/\alpha}$& $\alpha = 4$ & $b = \left(\frac{4}{3}\right)^{1/4}$ & $\frac{1}{2} + \frac{1}{\pi} \arcsin\left(\frac{t}{b}\right) + \frac{t}{4} f(t) $ & $\frac{1}{\pi}\sqrt{b^2-t^2}\left( 2 t^2 + b^2 \right)$ \\
                               & $\alpha = 5$ & $b = \left(\frac{3\pi}{8}\right)^{1/5}$ & $\frac{1}{2} + \frac{1}{\pi}\arcsin\left(\frac{t}{b}\right) + \frac{t}{5} f(t)$ & $\frac{5}{3 \pi^2} \left[b \left(3 t^2 + 2 b^2\right) \sqrt{b^2 - t^2} + 3 t^4 \log\left[\frac{b + \sqrt{b^2 - t^2}}{|t|}\right]\right]$ \\
    %$a(\alpha) = 2^{1/\alpha}$ & $\alpha = 6$ & $b = \left(\frac{16}{15}\right)^{1/6}$ & $F(t) = \frac{1}{2} + \frac{1}{\pi} \arcsin\left(\frac{t}{b}\right) + \frac{t \sqrt{b^2-t^2}}{16 \pi} \left( 8 t^4 + 4 b^2 t^2 + 3 b^4 \right)$ & $f(t) = \frac{3}{8\pi} \sqrt{b^2 -t^2} \left(8 t^4 + 4 b^2 t^2 + 3 b^4\right)$ \\
    & $\alpha = 6$ & $b = \left(\frac{16}{15}\right)^{1/6}$ & $\frac{1}{2} + \frac{1}{\pi} \arcsin\left(\frac{t}{b}\right) + \frac{t}{6} f(t)$ & $\frac{3}{8\pi} \sqrt{b^2 -t^2} \left(8 t^4 + 4 b^2 t^2 + 3 b^4\right)$ \\
  \bottomrule
  \end{tabular}
  \renewcommand{\arraystretch}{1}
  \renewcommand{\tabcolsep}{12pt}
}
\end{center}
\caption{Explicit formulae for special cases of Hermite/exponential-type asymptotic distributions \cite{ullman_orthogonal_1980}. See Table \ref{tab:summary-table}. The distribution function $F(t)$ for $\alpha=2$ is plotted in the right-hand pane of Figure \ref{fig:univariate-leja-distribution}. The densities $f(t)$ for these tabulated values of $\alpha$ are plotted in Figure \ref{fig:asymptotic-density}.}\label{tab:hermite-table}
\end{table}

\begin{figure}
\begin{center}
  \iftoggle{arxiv}{
    \resizebox{\textwidth}{!}{\includegraphics{asymptotic-density}}
  }{
    \resizebox{\textwidth}{!}{\includegraphics{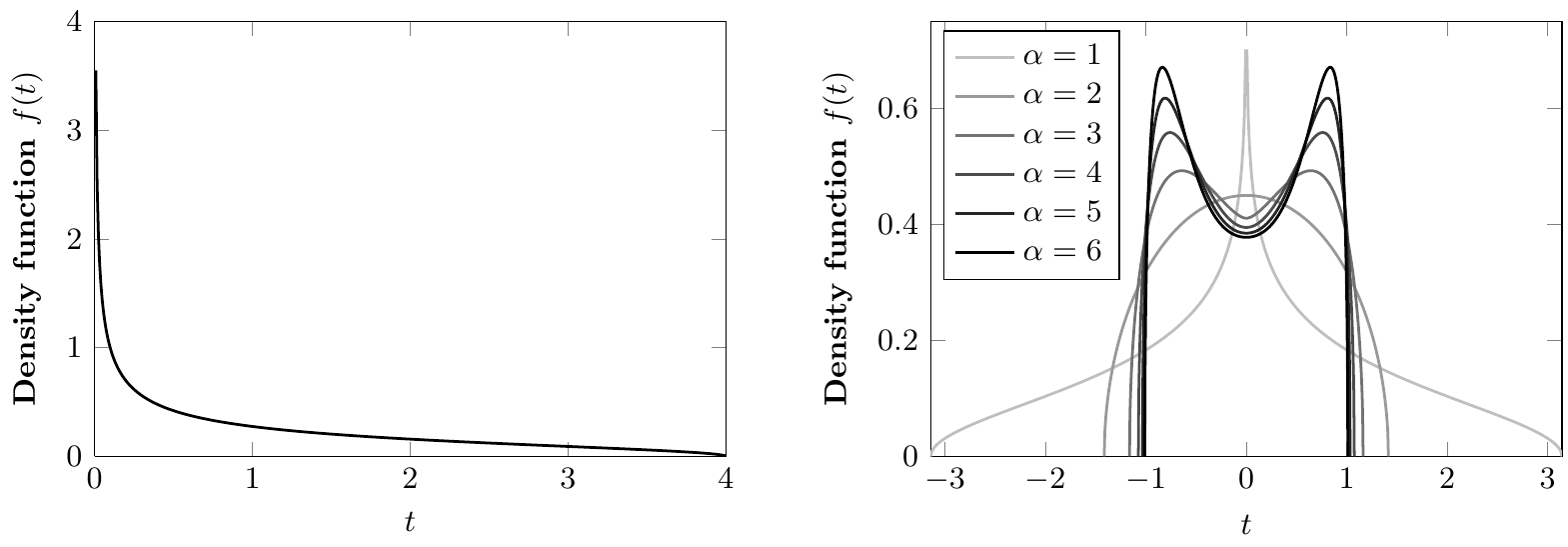}}
  }
\end{center}
\caption{Plots of asymptotic densities for contracted weighted Leja (and Gauss) points. Left: Limiting density for the Laguerre weight $w(z) = \exp(-z)$. Right: Limiting density for the exponential weight $w(z) = \exp(-|z|^\alpha)$ for various $\alpha$. See Table \ref{tab:summary-table} for a summary and Table \ref{tab:hermite-table} for explicit formulas.}\label{fig:asymptotic-density}
\end{figure}

We note that by affine scaling, the limits of weighted Leja sequences in Tables \ref{tab:summary-table} and \ref{tab:hermite-table} for other types of weights are readily derivable. I.e., suppose we have a new weight $W$ defined on a new parameter $Z$ that results from affine scaling of a canonical $w$, $z$ pair from Theorem \ref{thm:main-leja-result}:
\begin{align*}
  Z &= T(z) \triangleq A z + B, & W(Z) &= w\left(T^{-1}(Z)\right),
\end{align*}
for some $A,B \in \R$. Then the new limiting density $g(s)$ and distribution $G(s)$ for a $W$-weighted Leja sequence on the $Z$ domain $T(I)$ can be expressed in terms of the weighted Leja asymptotics in Tables \ref{tab:summary-table} and \ref{tab:hermite-table}:
\begin{align*}
  g(s) &= \frac{1}{A} f\left(T^{-1}(s)\right), & G(s) = F\left(T^{-1}(s)\right)
\end{align*}
for $s \in T(I)$. The contraction factor for $Z$ will be the same as it was for $z$.

We prove Theorem \ref{thm:main-leja-result} by leveraging a significant result from potential theory: nodal sets that are `asymptotically weighted Fekete' distribute according to the weighted potential equilibium measure \cite{saff_logarithmic_1997,totik_fast_1994}. We prove that contracted versions of weighted Leja sequences are asymptotically weighted Fekete and essentially obtain (our novel contribution to) Theorem \ref{thm:main-leja-result} as a corollary.

Let $V(\xi_0, \ldots, \xi_N)$ denote the modulus determinant of the polynomial Vandermonde matrix $W$ on any array of points $\xi_0, \ldots, \xi_N$. $W$ has entries $(W)_{j,k} = \xi_j^{k-1}$ for $j,k = 0, \ldots N$. Given a weight function $\widetilde{v}$ and a positive integer $N$ define the maximum attainable value for the following weighted determinant: 
\begin{align*}
  \delta^{(N)}_{\widetilde{v}} = \left[ \max_{(\xi_0, \ldots, \xi_N) \in I^{N+1}} V(\xi_0, \ldots, \xi_N) \prod_{j=0}^N \widetilde{v}^N(\xi_j) \right]^{2/(N^2+N)}
\end{align*}
For any fixed $N$, a cardinality-$N$ point set that achieves the maximum weighted determinant under the bracket is called a weighted Fekete set. It is known that the behavior of this maximum determinant value has a finite limit, the $\widetilde{v}$-weighted transfinite diameter of $I$:
\begin{align*}
  \delta_{\widetilde{v}} = \lim_{N \rightarrow \infty} \delta^{(N)}_{\widetilde{v}}.
\end{align*}
Any set of points whose asymptotic determinant limits to the transfinite diameter is called a set of asymptotically weighted Fekete points. In the cases of Theorem \ref{thm:main-leja-result}, the $k_N$-contracted $w$-weighted Leja points we proposed in \eqref{eq:weighted-leja-objective} are asymptotically $\widetilde{v}$-weighted Fekete.
\begin{theorem}\label{thm:asymptotically-fekete}
  In all the cases of Theorem \ref{thm:main-leja-result}, the Leja sequence $z_n$ defined by \eqref{eq:weighted-leja-objective} produces a set of points whose $k_N$-contraction is asymptotically weighted Fekete:
  \begin{align}\label{eq:leja-asymptotically-fekete}
    \lim_{N\rightarrow\infty} \left[ V(z_{0,N}, \ldots, z_{n,N}) \prod_{n=0}^N \widetilde{v}^N(z_{n,N}) \right]^{2/(N^2+N)} &= \delta_{\widetilde{v}}, & z_{n,N} = k_N z_n
  \end{align}
  where $\widetilde{v}$ is the weight function corresponding to the limit measure in Theorem \ref{thm:main-leja-result}. 
\end{theorem}
Theorem \ref{thm:asymptotically-fekete} is a stronger result than Theorem \ref{thm:main-leja-result} (see Lemma \ref{lemma:fekete-limit}) and is the result that we spend the most effort proving. Once this is established, it is well-known that asymptotically weighted Fekete points distribute according to the weighted equilibrium measure $\mu_{\widetilde{v}}$. See Section \ref{sec:proofs} for details and the proof.

\subsection{Quadrature with Leja sequences}
The construction of Leja points is motivated mainly by interpolation; however quadrature/cubature in a multidimensional sparse grid framework is very desirable. To this end, one may simply explicitly integrate an interpolant on Leja points to construct a quadrature rule. 

Consider a weighted Leja sequence $z_n$ constructed using \eqref{eq:weighted-leja-objective}. We need only integrate the interpolant constructed from data on the $z_n$. Let $\{p_n\}$, $n=0, \ldots, N-1$ denote family of polynomials orthonormal under $w(z)$. There are two observations we need: (i) if we assume that $w$ is a probability density function, then $p_0 \equiv 1$, and (ii) even if $w$ is known only empirically and does not have a representation in terms of classical functions, there are simple and accurate methods to construct the $p_n$ in one dimension \cite{kautsky_calculation_1983,gautschi_orthogonal_2004,narayan_computation_2012}. Given data $f_n$ we wish to interpolate at the sites $z_n$, so we seek the coefficients $c_n$ solving the linear problem
\begin{align*}
  \bV \bc &= \boldf, & V_{n,m} = p_{m-1}(z_n).
\end{align*}
Since $\bc = \bV^{-1} \boldf$, and 
\begin{align*}
  \int_\Omega \sum_{n=1}^N c_n p_{n-1}(z) \omega(z) \dx{z} = c_1 \int_\Omega p_0(z) \omega(z) \dx{z} = c_1,
\end{align*}
then we immediately conclude that the first row of the matrix $\bV^{-1}$ gives us quadrature weights $w_n$ defining the Leja polynomial quadrature rule
\begin{align*}
  Q^1_{N-1} f \simeq \sum_{n=1}^N w_n f(z_n).
\end{align*}
The superscript `1' indicates that this quadrature rule applies to one-dimensional functions, and the subscript $N-1$ refers to the `level' of the quadrature rule; both of these indicators unnecessary at the moment, but are meaningful in coming sections.

Naturally we wish to understand whether the Leja quadrature rules are useful in one dimension before proceeding to use them in higher-dimensional situations. We first verify that the quadrature rules are stable. \rev{The relative condition number of the quadrature rule $Q^1_{N-1}$ is given by the $\infty$-norm of the $1 \times N$ matrix $\mathbf{W}$ with entries $W_{1,n} = w_n$.} Thus, the condition number of the quadrature rule is 
\begin{align*}
  \kappa^1_{N-1} = \frac{ \sum_{n=1}^N |w_n| }{ \sum_{n=1}^N w_n} = \sum_{n=1}^N |w_n|,
\end{align*}
where the last equality holds under the assumption that the $p_n$ are orthonormal with respect to a probability density function $w$. The metric $\kappa^1$ indicates the presence of negative weights, which make the computation susceptible to catastrophic cancellation. The left-hand pane of Figure \ref{fig:quadrature-conditioning} graphs $\kappa^1_{N-1}$ for three choices of $w$: the uniform density on $[-1,1]$, an oscillatory weight $w$ on the same domain, and finally the Gaussian density function $w \propto \exp(-z^2)$ on $\R$. We see that the quadrature rules are all relatively well-conditioned. 

\begin{figure}
\begin{center}
  \iftoggle{arxiv}{
    \resizebox{\textwidth}{!}{\includegraphics{quadrature-conditioning}} 
  }{
    \resizebox{\textwidth}{!}{\includegraphics{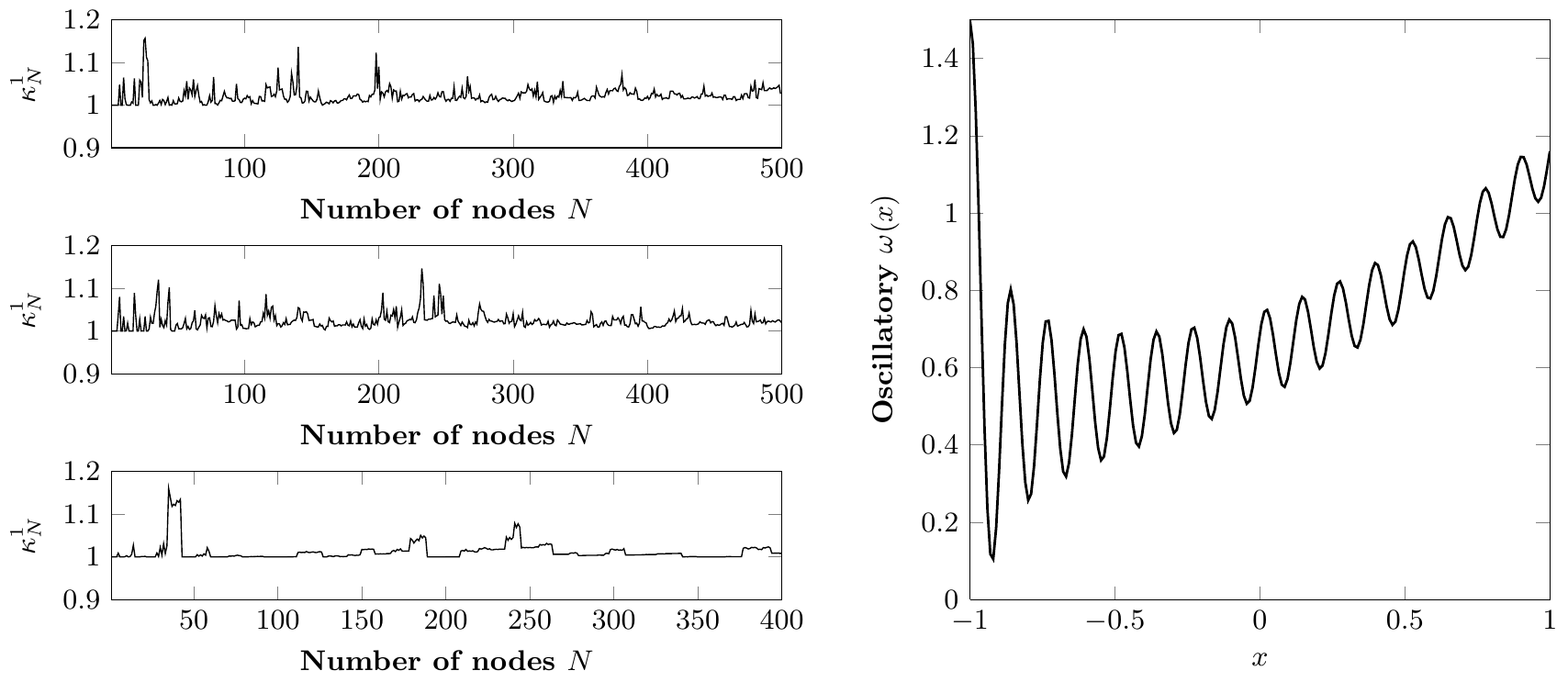}} 
  }
\end{center}
\caption{Left: Absolute condition number for Leja interpolatory quadrature rules for various weight functions $\omega$: (top) $\omega \propto 1$ the uniform density, (middle) $w(z) \propto \frac{1}{2} I_0(1+z) + J_0(50 + 50 z)$, where $J_0$ ($I_0$) is the (modified) Bessel function of the first kind, (bottom) $w(z) \propto \exp(-z^2)$ the Gaussian density. Right: plot of the second weight function, $w(z) \propto \frac{1}{2} I_0(1+z) + J_0(50 + 50 z)$.}
\label{fig:quadrature-conditioning}
\end{figure}

The accuracy of the rules are considered in the following examples:
\begin{align*}
  f_1(z) &= \frac{1}{1 + 101 \left(z - \frac{\pi}{4}\right)^2} &
  f_2(z) &= \cos \left( 1 + 100 z^3\right)
\end{align*}
We perform one-dimensional global interpolation and refinement for these functions using the (uniform) Leja grids, nested Clenshaw-Curtis grids, and \rev{Legendre}-Gauss-Patterson grids. Figure \ref{fig:d1-runge} shows results in the discrete maximum norm and the quadrature error. If $\widetilde{f}$ is the interpolatory approximation, then on a $10^4$-sized Clenshaw-Curtis grid $x_n$ with weights $v_n$, these metrics are defined as
\begin{align*}
  \sqrt{ \sum_n v_n \left(f_1(x_n) - \widetilde{f}_1(x_n) \right)^2 } & & \textrm{(Discrete $\ell^2$ error)} \\
  \max{ \left| f_1(x_n) - \widetilde{f}_1(x_n) \right| } & & \textrm{(Discrete maximum error)} \\
  \left| \sum_n v_n \left( f_1(x_n) - \widetilde{f}_1(x_n)\right) \right| & & \textrm{(Quadrature error)}
\end{align*}
The left-hand pane of Figure \ref{fig:d1-runge} shows that the Leja grid is no less accurate than any of them in the maximum norm (being as accurate as the Clenshaw-Curtis grid).
A discrete $\ell^2$ error metric behaves similarly.
The Leja sequence performs noticeably worse than the other two for the quadrature metric. We sacrifice quadrature accuracy in order to gain some dexterity in high-dimensional refinement: with the Clenshaw-Curtis or \rev{Legendre}-Gauss-Patterson grid every refinement doubles the size of the (univariate) rule, whereas with the Leja procedure we can stop refinement at any size we choose. \rev{That the Gauss-Patterson grid performs so poorly for the maximum norm approximation can be explained by the fact that Gauss-Patterson nodes are constructed only to obtain a high degree of polynomial integration, not for interpolatory approximation. Gauss-Patterson grids that are formed for nested quadrature are not necessarily good for interpolation. This can be seem by comparing the left- and right-hand panels in Figure \ref{fig:d1-runge}.}

\begin{figure}
\begin{center}
  \iftoggle{arxiv}{
    \resizebox{\textwidth}{!}{\includegraphics{d1-runge}}
  }{
    \resizebox{\textwidth}{!}{\includegraphics{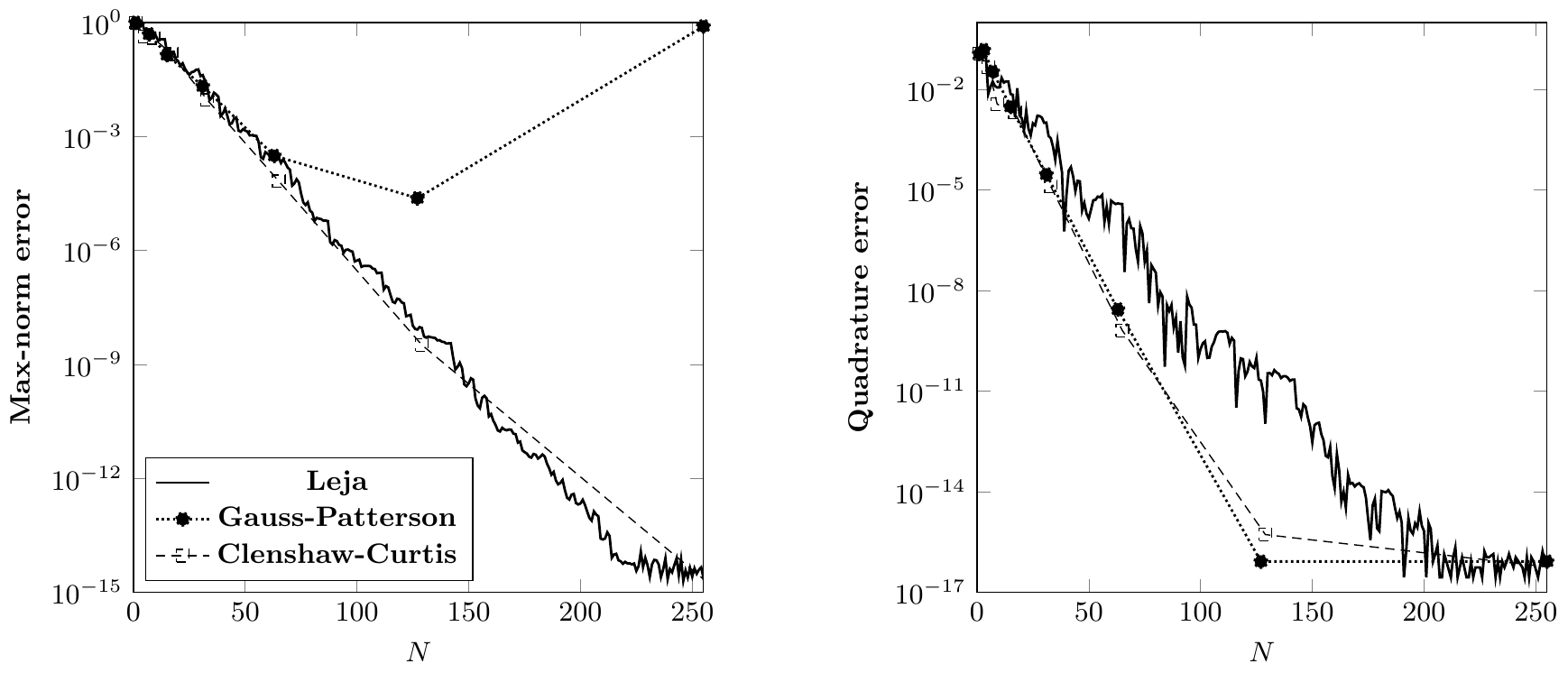}}
  }
\end{center}
\caption{Left: Maximum errors for the one-dimensional Runge function $f_1$ with Leja, Clenshaw-Curtis, and Gauss-Patterson nested refinements. Right: Maximum hierarchical surplus at each level.}
\label{fig:d1-runge}
\end{figure}

\begin{figure}
\begin{center}
  \iftoggle{arxiv}{
    \resizebox{\textwidth}{!}{\includegraphics{d1-cos-surplus}}
  }{
    \resizebox{\textwidth}{!}{\includegraphics{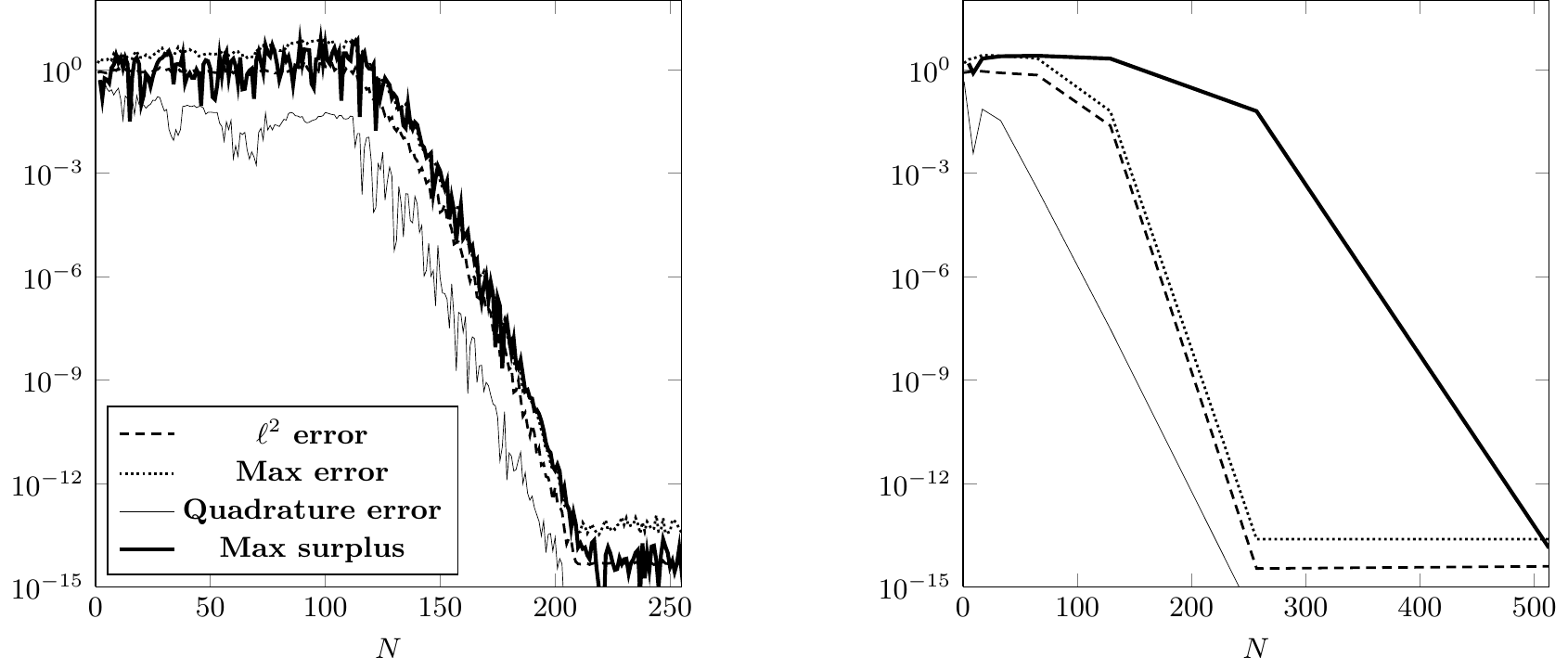}}
  }
\end{center}
\caption{Left: Errors and the hierarchical surplus at each level for a one-dimensional Leja refinement of oscillatory function $f_2$. Right: Errors and surplus using a Clenshaw-Curtis refinement.}
\label{fig:d1-cos-surplus}
\end{figure}

The relatively good behavior of a Leja grid is not useful unless we have a good error metric from the hierarchical surplus. Now consider function $f_2$, and the results in Figure \ref{fig:d1-cos-surplus}, where we show the behavior of the surplus versus all the error metrics. The left-hand pane shows that the maximum Leja surplus is an excellent indicator of error in all three norms. The right-hand pane shows similar results for the Clenshaw-Curtis grid, but two observations are apparent: first, if using the surplus as a refinement technique, the surplus when $N = 257$ does not accurately reflect the actual error at the next level $N = 513$. This is true even if one were to employ a type of Richardson extrapolation to estimate the error. Second, the surplus indicator is a very conservative estimate of the error. In such a case, it is likely that we will refine more than is necessary in order to obtain an approximation. Both of these observations do not hold for the Leja surplus, which is sharper estimate of the error, and may be refined with arbitrary size.

\subsection{Barycentric Interpolation}
A numerically robust method for computing polynomial interpolants is furnished by the Barycentric interpolation formula (see \cite{berrut_barycentric_2004} for an accessible introduction). This method is computationally efficient with respect to computing and evaluating the interpolant, and is stable so long as the interpolation problem itself is stable. 

On a set of nodes $z_1, \ldots, z_N$, the degree-$(N-1)$ polynomial interpolant of a function $f(z)$ with data $f_n = f(z_n)$ is given by 
\begin{align}\label{eq:barycentric-formula}
  p(z) = \sum_{n=1}^N f_n \ell_n(z) = \frac{\sum_{n=1}^N \frac{b_n f_n}{z - z_n}}{\sum_{n=1}^N \frac{b_n}{z-z_n}},
\end{align}
where $\ell_n(z)$ are the cardinal Lagrange interpolation basis and $b_n$ are the Barycentric weights, both defined as
\begin{align*}
  \ell_n(z) &= \prod_{m\neq n} \frac{z-z_m}{z_n - z_m}, &
  b_n & = \left[ \prod_{m\neq n} (z_n - z_m)\right]^{-1}.
\end{align*}
The Barycentric formula \eqref{eq:barycentric-formula} allows evaluation of the interpolant $p(z)$ in only $\mathcal{O}(N)$ operations once the weights $b_n$ are precomputed. In addition, the symmetry of this formula allows one to renormalize all the weights $b_n$ by the same constant without affecting the interpolant. It is known that when the interpolation nodes $z_n$ correspond to a well-conditioned interpolation operator, then the weights $b_n$ all have comparable magnitude, leading to a well-conditioned numerical procedure. 

With Leja sequences, we are essentially interested in $v$-weighted polynomial interpolation. (Recall \eqref{eq:square-root-weight}.) This means that while we want to produce a polynomial interpolant $f$, we do so by interpolating the function $v f$ using $v$-weighted polynomials. With this in mind, it is straightforward to show that the $v$-weighted analogue of \eqref{eq:barycentric-formula} is the following polynomial formula:
\begin{align}\label{eq:weighted-Barycentric-formula}
  p(z) = \frac{1}{v(z)} \sum_{n=1}^N v_n f_n \ell^v(z) = \frac{\sum_{n=1}^N \frac{b^v_n v_n f_n}{v - v_n}}{\sum_{n=1}^N \frac{b^v_n v_n}{z - z_n}},
\end{align}
where $v_n \triangleq v(z_n)$, $\ell^v$ are $v$-weighted Lagrange polynomials, and $b^v_n$ are $v$-weighted Barycentric weights:
\begin{align*}
  \ell^v_n(z) &= \frac{v(z)}{v_n} \ell_n(z), &
  b^v_n & = \frac{b_n}{v_n}. 
\end{align*}
For a sequence of Leja nodes generated according to \eqref{eq:weighted-leja-objective}, we use the Barycentric weights $b^v_n$ given above to perform interpolation. We observe in practice that with this normalization that the weights $b^v_n$ are all of comparable magnitude, just as we expect them to be for a well-conditioned interpolation problem. 

Note that we do not necessarily avoid any troublesome numerical computations in the reformulated case \eqref{eq:weighted-Barycentric-formula}; we have merely recast the problem into one that appears numerically well-conditioned. The actual process of interpolating $f$ by an unweighted polynomial on an unbounded domain will still be mathematically ill-conditioned.

}{
  
}

\section{Sparse grids with univariate Leja rules}\label{sec:sparse-grids}
\iftoggle{arxiv}{
  The approximation of a quantity depending on a finite number of Euclidean-like parameters is difficult when the parametric dimension $d$ is large.
 It is well-known that an approximation to an $r$-times differentiable function converges with a rate of $\mathcal{O}(N^{-r/d})$ \cite{bellman_dynamic_2003}. 
Let $f: \Gamma_{\bxi} \rightarrow \R$ for $\Gamma_{\bxi} \subset \R^d$ be a function that we wish to approximate and $\Gamma_{\bxi}$ be the domain of 
the possibly high-dimensional parameter $\bxi$ upon which the function depends. Spatial and temporal variables are modeled separately.
For simplicity we assume that $\Gamma_{\bxi}$ is an isotropic tensor-product domain, so that the one-dimensional restricted variables $\param_k$ with $\bxi = \left(\param_1, \ldots, \param_d\right)$ all take values on the same restricted one-dimensional space.

When $f$ is approximated via a sampling procedure, the curse of dimensionality is readily apparent: let $\Xi^1 \subset \R$ be an $n$-point nodal set in one dimension with associated quadrature weights $w_k$. A tensorization of this quadrature rule over $d$ dimensions yields the nodal set
\begin{align}\label{eq:tensor-construction}
  \Xi^d &= \bigotimes_{j=1}^d \Xi^1, & N = |\Xi^d| = n^d.
\end{align}
The growth of the size of the tensorized set $\Xi^d$ usually makes it infeasible for usage in high-dimensional approximation methods. (This is true both in cases when $d$ is not small and fixed and $n$ is increased, or when $n$ is fixed and $d$ is increased.) There are alternatives to tensor constructions, but approximation with any space-filling design requires $\mathcal{O}(n^d)$ samples, where $n$ is the number of samples `per dimension'. 

An alternative to space-filling designs is the popular \textit{sparse grid}, so named because of its geometrically dispersed distribution in $\Gamma_{\bxi}$. 
Like tensor constructions, sparse grids tensorize univariate nodal arrays, but sparse grids also attempt to delay the impact of the curse of dimensionality by taking only 
certain combinations of tensor products. We delay introduction of sparse grids until Section \ref{sec:sparse-grids}; for now we concentrate on motivating our choice of univariate rule: Leja sequences.

\subsection{Common univariate rules}
The sparse grid construction requires specification of a univariate grid $\Xi^1_l$, for $l = 0, 1, \ldots$. Several choices for these univariate grids work well, and among the most popular are a Clenshaw-Curtis (CC) grid, or grids associated with Gauss quadrature rules \cite{gerstner_numerical_1998}. For concreteness, we consider the one-dimensional finite interval $[-1,1]$. Then, for example, we may choose CC and we have the following grid for any level $l$:
\begin{align*}
  \Xi_{l} &= \left\{\param_{l,n}\right\}_{n=1}^{N_l}, & \param_{l, N_l+1-n} &= \cos \left(\frac{(n-1) \pi}{N_l-1}\right),
\end{align*}
for $n = 1, \ldots, N_l$, and $N_l = 2^l + 1$. This particular choice for $\Xi^1_l$ is popular for two reasons: (i) $\Xi_l \subset \Xi_{l+1}$ allowing for \textit{hierarchical} approximation and adaptive refinement with as few model evaluations as possible, and (ii) the sequence $\Xi_l$ is known to be both an excellent interpolatory and quadrature grid. Of course, one apparent concern is that $\left| \Xi_{l+1} \right| - \left|\Xi_l\right| = N_{l+1} - N_l = 2^l$, which grows exponentially with the level $l$. This means that each stage of refinement requires addition of a large number of points. In general, a large number of nodes is not necessarily adverse so long as the resulting grids have some optimality regarding, e.g., maximum degree of polynomial integration \cite{petras_smolyak_2003}.

An alternative univariate rule that is competitive for quadrature purposes is the Gauss-Patterson grid \cite{patterson_optimum_1968} wherein one constructs a grid that is a subset of a given Gauss quadrature grid, and satisfies some polynomial integration optimality conditions. However, it is not always possible to construct such grids depending on (i) the cardinality of the subset $X_l$ and (ii) the weight function $\omega$. Even when such construction is possible, construction of the Gauss-Patterson grid requires implementation of a nontrivial algorithm, and it is frequently easier to precompute and store the grids, making the method inflexible with respect to the choice of density $\omega$.

There are several locally adaptive strategies for sparse grids that are also successful in combating the curse of dimensionality~\cite{jakeman_localuq_2013,ma_anova_sg_2010}. One method that has enjoyed recent success is the locally adaptive, high order, generalized sparse grid construction~\cite{jakeman_localuq_2013}. In this setup, one uses a high-order Lagrange polynomial basis as the univariate building block for a local high-order polynomial approximation; because the approximation is local, targeted adaptive strategies that utilize the grid hierarchical surpluses are naturally applicable and effective. However, the adaptation is usually (locally) uniform and it is well-known that high-order polynomial approximation on a uniform grid raises computational challenges.

We propose use of (weighted) Leja sequences as univariate building blocks for an adaptive Smolyak sparse grid constructions. Leja sequences can easily add an arbitrary number of samples at each stage, and have good interpolatory and quadrature properties, making them excellent ingredients for the Smolyak algorithm. Leja sequences have been used a sparse grid building blocks before: \cite{calvi_lebesgue_2011,calvi_lagrange_2012}, but we believe this is the first investigation into adaptive hierarchical approximations for high-dimensional approximation.

%$\Xi^1_l$ denote a univariate set of nodes on $\Gamma_{\xi_1}$ for $l = 0, 1, \ldots$, where $\left|X^1_l\right| = n_i$. The index $l$ is the \textit{level} of the univariate nodal set, and $n_l$ is assumed be monotone increasing in $l$. Thus, higher levels correspond to denser grids. A level-$\ell$ isotropic sparse grid construction uses a combination of tensorizations:
%\begin{align}\label{eq:smolyak-grid}
%  \Xi^d_\bl = \bigcup_{|\bl|=\ell-d+1}^\ell \bigotimes_{q=1}^d \Xi^1_{l_q},
%\end{align}
%where $\bl = (l_1, \ldots, l_d) \in \N^d$ is a multi-index with magnitude $|\bl| = \sum_{q=1}^d l_q$. Unlike \eqref{eq:tensor-construction}, the Smolyak construction uses tensorizations that attempt to control the total number of nodes. 

Having discussed univariate Leja sequences at length in Section \ref{sec:leja}, we may now construct standard sparse grids using Leja sequences as building blocks. Let $\bXi = \left(\param_1, \ldots, \param_d\right) \in I_{\bxi} \subseteq \R^d$ be a random variable with probability density function $\omega(\bxi): I_{\bxi} \rightarrow \R$, and assume that the components of $X$ are mutually independent so that $I_{\bxi}$ is a tensor-product domain 
\begin{align*}
  I_{\bxi} = \bigotimes_{j=1}^d I_{\param_j}. 
\end{align*}
We let $\omega_i: I_{\param_j} \rightarrow \R$ denote the marginal PDF of $\param_i$, so that $\omega(\bxi) = \prod_{j=1}^d \omega^j(\param_j)$. Sparse grids~\cite{bungartz04} approximate $\qoih$ via a 
weighted linear combination of basis functions
\begin{equation}\label{eq:sgInterpolantsum}
\intp{n}{\qoih} := \qoinh = \sum_{k=1}^n v_{k}\, \Psi_{k}(\bxi)
\end{equation}
%where, without loss of generality, the dependence of $\solh$  on $\bx$ and $t$ has been dropped 
%to simplify the remainder of this discussion. 
The approximation is constructed on a set of anisotropic grids $\Xi_\bl$ on
the domain $\dom$ 
where $\bl=(l_1,\ldots,l_d)\in \mathbb{N}^d$ is a multi-index denoting the level
of refinement 
of the grid in each dimension. These rectangular grids are Cartesian product
of nested one-dimensional grid points 
$\Xi_{l}=\{\xi_{l,i}:i<0\le i \le m_{l}\}$
\[ \Xi_\bl = \Xi_{l_1} \times \cdots \Xi_{l_d}
\]
% The one-dimensional grids $\Xi_{l}$ are typically the Gauss points associated with the family of polynomial orthogonal under the weight $\omega$. These Gaussian rules are in general are not
% nested so sometimes nested quadrature rules are used to minimize the number of grid points needed to grow the sparse grid from one level to the next. 
% Examples of nested quadrature rules are the Gauss-Patterson and Clenshaw-Curtis rules which which are associated with uniform random variables and the Genz-Keister quadrature rule used for Gaussian random variables\cite{genz_fully_1996}.
The number of points $m_{l}$ of a one-dimensional 
grid of a given level is dependent on the growth rate of the quadrature rule chosen.

The multivariate basis functions $\Psi_k$  are a tensor product of one
dimensional basis functions. 
Adopting the multi-index notation used above we have
\begin{equation}
\label{eq:multi-dimensional-basis} 
\Psi_{\bl,\bi}(\bxi)=\prod_{n=1}^d \psi_{\ell_n,i_n}(z_n)
\end{equation}
where $\bi$ determines the location of a given grid point. There is a one-to-one 
relationship between $\Psi_k$
in~\eqref{eq:sgInterpolantsum} and $\Psi_{\bl,\bi}$ and each $\Psi_{\bl,\bi}$
is 
uniquely associated with a grid point 
$\bxi_{\bl,\bi}=(\xi_{\ell_1,i_1},\ldots,\xi_{\ell_d,i_d})\in \Xi_\bl$. 
Many different one-dimensional basis functions $\psi_{\ell_n,i_n}(\xi_n)$ 
can be used. In the following we employ one-dimensional Lagrange polynomials for the functions $\psi_{\ell_n,i_n}$.

The multi-dimensional basis~\eqref{eq:multi-dimensional-basis} spans 
the discrete space $V_\bl \subset L^2(\dom)$%{\color{red} CAN WE RELATE THIS TO THE SPACE OF $\qoih$}
\begin{align*}
  V_\bl&=\text{span}\left\{ \Psi_{\bl,\bi}\,:\bi\in\cK_\bl\right\} & 
  \cK_\bl&=\{\bi \in \N_0^d :i_k=0,\ldots,m_{l_k}\,,k=1,\ldots,d\}
\end{align*}
These discrete spaces can be further decomposed into 
hierarchical difference spaces
\[
 W_\bl={V_\bl} \setminus V_{\bl}\, \bigoplus_{n=0}^d V_{\bl-\be_n}
\]
The subspaces $W_\bl$ consists of all basis functions $\Psi_{\bl,\bi} \in V_\bl$
which are not included in any of the spaces $V_\bk$ smaller than $V_\bl$, 
i.e. with $\bk<\bl$ with $<$ the lexicographic partial ordering on multi-indices.
These hierarchical difference spaces can be used to decompose the input space
such that %for a functions space $V$. {\color{red} CAN WE RELATE THIS TO THE SPACE OF $\qoih$}
\begin{equation*}
\label{eq:vspace}
V_\bl=\bigoplus_{\bk\le \bl}W_\bl\quad\text{and}\quad
 L^2(\dom)=\bigoplus_{k_1=0}^{\infty}\cdots\bigoplus_{k_d=0}^{\infty}
W_\bk=\bigoplus_{\bk\in\mathbb{R}^d}W_\bk
\end{equation*}

For numerical purposes we must truncate the number of difference spaces used to
construct $V$.  Traditional isotropic 
sparse grids can be obtained by all hierarchical subspaces $W_\bl$ with
and index set that satisfy
\begin{equation}
\label{eq:sg-isotropic-index-set}
 \cL = \{\bl:|\bl|_1\le l\}
\end{equation}

Given a truncation, such as the \textit{a priori} one above or one which has been determined adaptively,
$\qoih$ can be approximated by
\begin{align}\label{eq:sgInterpolant}
  \qoinh&=\sum_{\bl\in\cL}  \qoih_\bl,& 
  \qoih_\bl&= \sum_{\bi\in\cI_\bl} v_{\bl,\bi}\, \Psi_{\bl,\bi}(\bxi)
\end{align}
where $\cI_\bl = \{\bi:\Psi_{\bl,\bi}\in W_\bl \}$.

Here we note that the $v_{\Bell,\bj}$ are the coefficient values of the hierarchical product 
basis, also known as the hierarchical surplus. The surpluses are simply 
the difference between the function value and the sparse grid approximation at
a point, not already in the sparse grid. That is 
\begin{align*}
  v_{\bl,\bj} &= \qoih(\bxi_{\bl,\bi}) - \qoinh(\bxi_{\bl,\bi}),&
  \cL \cap \bl &= \emptyset
\end{align*}
The particular choice of sparse grid in this paper is one constructed with univariate hierarchical Leja points: the sequence of points that we use for each dimension to evaluate the surplus and construct the interpolant is a univariate Leja sequence. In this way, the point sets $\bxi_{\bl,\bi}$ are nested, and the number of points to add at each level can be as small or large as we wish. (I.e. the Leja choice allows a great deal of granularity for refinement.)

\subsection{Dimension adaptivity}

\rev{The dimensional adaptivity of our algorithm in this section is based on the idea presented in \cite{gerstner_numerical_1998}}. We begin with a low-level isotropic sparse grid approximation with a set of subspaces representing the current approximation $\cL$ and the set of active subspaces $\cA$ that indicate the levels for potential refinement.  Often $\cL=W_{\bzero}$ and $\cA=\{W_{\be_k},k=1\ldots,d\}$.  We then choose $W_\bl \in \cA$ with the largest error indicator $\gamma_\bl$ and refine that subspace. Here we define the error indicator $\gamma_\bl$ as
\begin{align}
\label{eq:dim-surplus-indicator}
  \gamma_\bl &= \int_{\Gamma_{\bxi}} (\qoih_\bl)^2 d\omega(\bxi) - \left(\int_{\Gamma_{\bxi}} (\qoih_\bl) d\omega(\bxi)\right)^2, & 
  \eta &= \sum_{\bl\in\cA} \gamma_\bl
\end{align}
The indicator $\gamma_\bl$ measures the contribution of the subspace $\bl$ to the variance of $f_n$ and the global indicator $\eta$ measures the contribution of all active subspaces
to the variance of $f_n$. These indicators are calculated by transforming the Lagrange interpolant on each {\it hierarchical} subspace into a Polynomial Chaos Expansion 
that is orthogonal to the (possibly mixed) distribution weight $\omega(\bxi)$. The cost of this transformation is linear in terms of the number of subspace points~\cite{buzzard2013}. The chosen subspace for refinement with index $\Bell$ is refined by adding all indices $W_\bk$ with 
$\bk = \bl+\be_n$, $n=1,\ldots,d$ that satisfy the following admissibility criterion
\begin{equation}
\label{eq:gsg_admissibility}
\bl-\mathbf{e}_k\in\cL\text{ for }1\le k\le d,\, l_k > 1
\end{equation}
The active set $\cA$ is then rebuilt by adding each subspace corresponding to the indices  from~\eqref{eq:gsg_admissibility}. This
process continues until a computational budget limiting the number of model samples
(grid points) is reached or a global error indicator drops below a predefined
threshold. Pseudo-code for the dimension adaptive algorithm is shown in Algorithm~\ref{alg:dim-adaptivity}. 

The \texttt{INDICATOR} and \texttt{TERMINATE} routines in Algorithm~\ref{alg:dim-adaptivity} control which subspaces are added
to the sparse grid via the use of a subspace error and global error metric. The indicators respectively provide
estimates of the contribution of a subspace to reducing the error in the
interpolant, and the error in the entire interpolant.
%Throughout this paper we will use the indicators
%\begin{equation}
%\label{eq:dim-surplus-indicator}
%\gamma_\bl = \int_{\Gamma_{\bxi}} (\qoih_\bl)^2 d\omega(\bxi) - \left(\int_{\Gamma_{\bxi}} (\qoih_\bl) d\omega(\bxi)\right)^2,\quad \eta = \sum_{\bl\in\cA} \gamma_\bl
%\end{equation}

\begin{algorithm}
\caption{\texttt{INTERPOLATE}[$f(\brv)$,$\cL$,$\cA$,$\tau$,$n$]$\rightarrow f_{n}$}
\label{alg:dim-adaptivity}
\begin{algorithmic}[1]
\STATE For a given $\cL$ the points in the sparse grid are $\Xi:= \bigcup_{\bl\in\cL} \Xi_\bl$.\\
\STATE The number of sparse grid points are $N=\#\Xi$\\
\WHILE {NOT \texttt{TERMINATE}[$\cA$,$N$,$\tau$,$n$]}
 \STATE $W:= \argmax_{W_\bl\in\cA} \gamma_\bl$ {\footnotesize\% Determine the subspace with the highest priority}
 \STATE $\cA:=\cA\setminus W$ {\footnotesize\% Remove $W$ from the active set}
 \STATE $\cL:= \cL \cup W$ 
 \STATE $\cJ:=$ \texttt{REFINE}[$W$,$\cL$] {\footnotesize\% Find all admissible forward neighbors of $W$}
 \STATE $\gamma_{\bl}:= $ \texttt{INDICATOR}[$W_\bl$]$\;\forall\; W_\bl\in\cJ$    {\footnotesize\% Calculate the priority of the neighbors}
 \STATE $\cA:= \cA \cup \cJ$ {\footnotesize\% Add the forward neighbors to the active index set}
\ENDWHILE
\end{algorithmic}
\end{algorithm}

}{
  
}

\section{Numerical examples}\label{sec:results}
\iftoggle{arxiv}{
  We consider several multidimensional examples below that compare the Smolyak-Leja algorithm with a more standard Clenshaw-Curtis-Smolyak algorithm. Effectively, we see that the Leja construction is competitive (usually superior) to Clenshaw-Curtis when an interpolation metric is used. However, they appear suboptimal when a quadrature metric is evaluated. This is not surprising as Leja sequences are constructed with the goal of interpolation and not necessarily for quadrature.

Throughout these examples we compute discrete $\ell_2$ errors $\varepsilon_{\ell_2}$ using 100,000 random samples taken in a Monte-Carlo fashion from the distribution of the input variable $\bxi$.  \rev{We also report absolute errors in the sparse grid mean $\varepsilon_{\mu}$ and variance $\varepsilon_{\sigma^2}$ where the exact moments are computed using a high-resolution sparse grid that was refined so that its $\ell_2$ error was in the order of machine precision. The error metric $\varepsilon_{\ell^2}$ is simply the discrete $\ell^2$ error (RMSE).}

\subsection{Random oscillator}
This section investigates the relative performance of the sparse grids 
when approximating the output from a model of linear oscillator subject to external forcing with six unknown 
parameters. That is,
\begin{equation}\label{eq:oscillator_ode}
\frac{d^2x}{dt^2}(t,\bz)+\gamma\frac{dx}{dt}+k x=f\cos(\omega t),
\end{equation}
subject to the initial conditions
\begin{equation}
x(0)=x_0,\quad \dot{x}(0)=x_1,
\end{equation}
where we assume the damping coefficient $\gamma$, spring constant $k$,
forcing amplitude $f$ and frequency $\omega$, and the initial
conditions $x_0$ and $x_1$ are all uncertain. We solve~\eqref{eq:oscillator_ode} analytically
to allow us to avoid consideration of discretization errors in our investigation.

Let us choose our quantity of interest to be the position $x(T)$ of the osciallator at $T=20$ seconds and
let $\bz=(\gamma,k,f,\omega,x_0,x_1)$ where $\gamma\in[0.08,0.12]$, $k\in[0.03,0.04]$, $f\in[0.08,0.12]$, $\omega \in [0.8,1.2]$, 
$x_0\in[0.45,0.55]$, $x_1\in[-0.05,0.05]$. For this choice of random parameters any parameter realization in $I_\bz$ will produce 
an underdamped harmonic oscillator.

Figure~\ref{fig:oscillator-d-6-rmse-convergence} compares the $\ell_2$ accuracy in the sparse grid interpolants obtained using Clenshaw-Curtis nodes 
and Leja nodes. Although both univariate rules have similar interpolation properties in one-dimension, the one-at-a-time nestedness of the Leja rule
produces, in this higher-dimensional setting, an approximation that is significantly more accurate than the approximation based upon the Clenshaw-Curtis quadrature rule.

\begin{figure}[ht]
\centering
\includegraphics[width=0.49\textwidth]{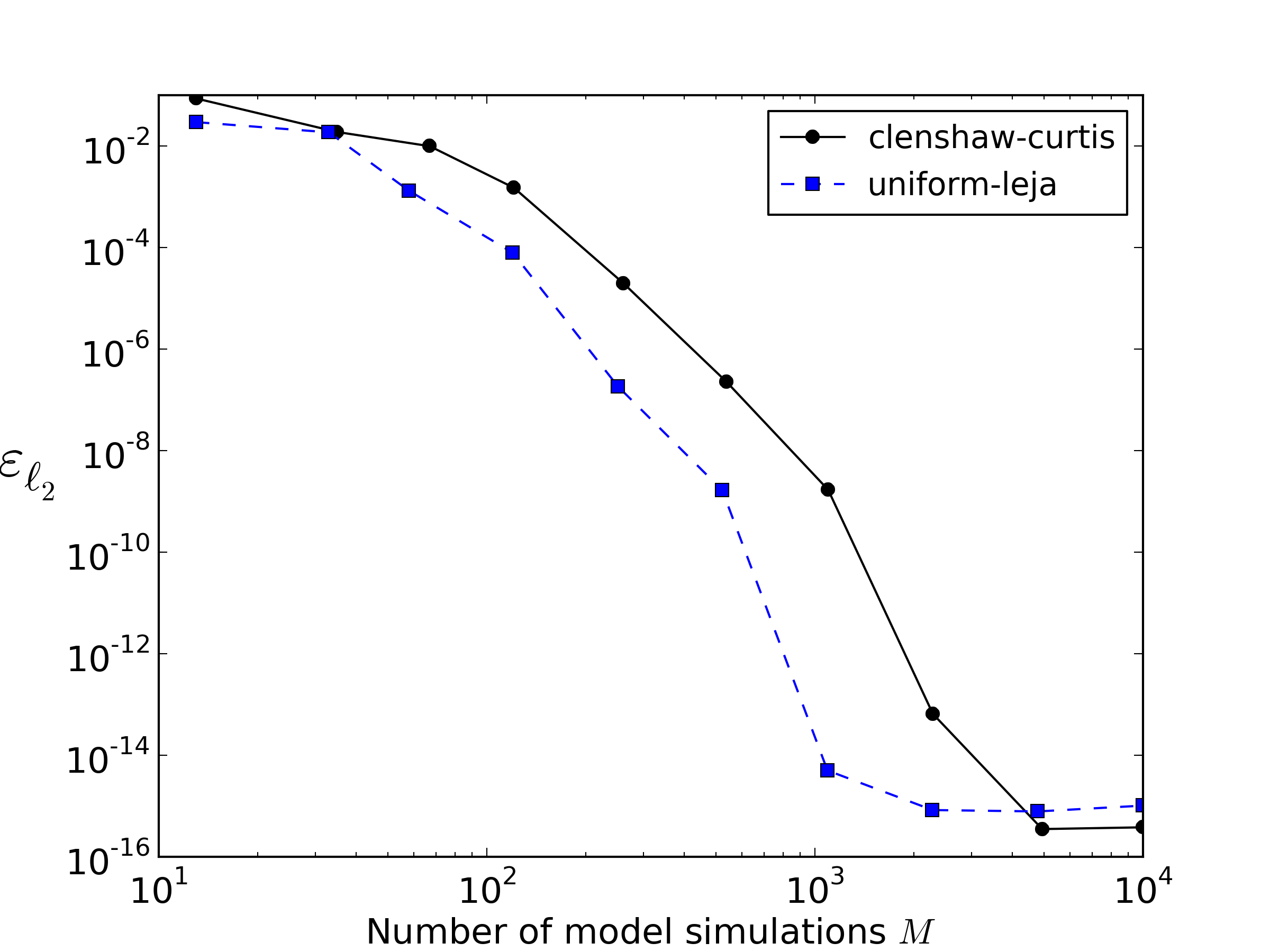}
\caption{Convergence of RMSE in the sparse grid approximation of the oscillator position with respect to the number of model evaluations.}
\label{fig:oscillator-d-6-rmse-convergence}
\end{figure}

\subsection{Borehole model}
For the next numerical demonstration, consider the following model of water flow through a borehole
\begin{equation}
\label{eq:borehole}
f(\rv) = \frac{2\pi T_u(H_u-H_l)}{\log(r/r_w)\left(1+\frac{2\pi T_u}{\log(r/r_w)r^2_wK_w}+
\frac{T_u}{T_l}\right)}
\end{equation}
%The response is water flow rate, in $m^3$/yr. 
where the unkonwn parameters are uniform random variables $\brv=(\rv_1,\ldots,\rv_8)$ with the following bounds: $\rv_1:=r_w \in [0.05, 0.15]$ (meters) denotes the radius of borehole, 
$\rv_2:=r \in [100, 50000]$ (meters) the radius of influence, $\rv_3:=T_u \in [63070, 115600]$ (meters$^2$/years) the transmissivity of upper aquifer,
$\rv_4:=H_u \in [990, 1110]$ (meters) the potentiometric head of upper aquifer,
$\rv_5:=T_l \in [63.1, 116]$ (meters$^2$/years) the transmissivity of lower aquifer,
$\rv_6:=H_l \in [700, 820]$ (meters) the potentiometric head of lower aquifer,
$\rv_7:=L \in [1120, 1680]$ (meters) the length of borehole, and 
$\rv_8:=K_w \in [9855, 12045]$ (meters/year) the hydraulic conductivity of borehole.

Figure~\ref{fig:borehole-d-40-rmse-convergence} compares the accuracy of sparse grids based upon the univariate Clenshaw-Curtis and Leja nodes. Again the
Leja interpolation sequence produces a more accurate interpolant for a given number of function evaluations, but suffers when evaluating quadrature quantities such as the mean shown in the left-hand pane. %However, we see that the realtive performance of the
%two approximations changes when we measure the quadrature error as shown in Figure~\ref{fig:}. 
In this case as with many others, the %but not all, in our experience the 
Clenshaw-Curtis quadrature rule produces a more accurate estimate of the mean of the function. This statement is consistent with the one-dimensional results shown in 
Figure~\ref{fig:d1-runge}.

\begin{figure}[ht]
\centering
\includegraphics[width=0.49\textwidth]{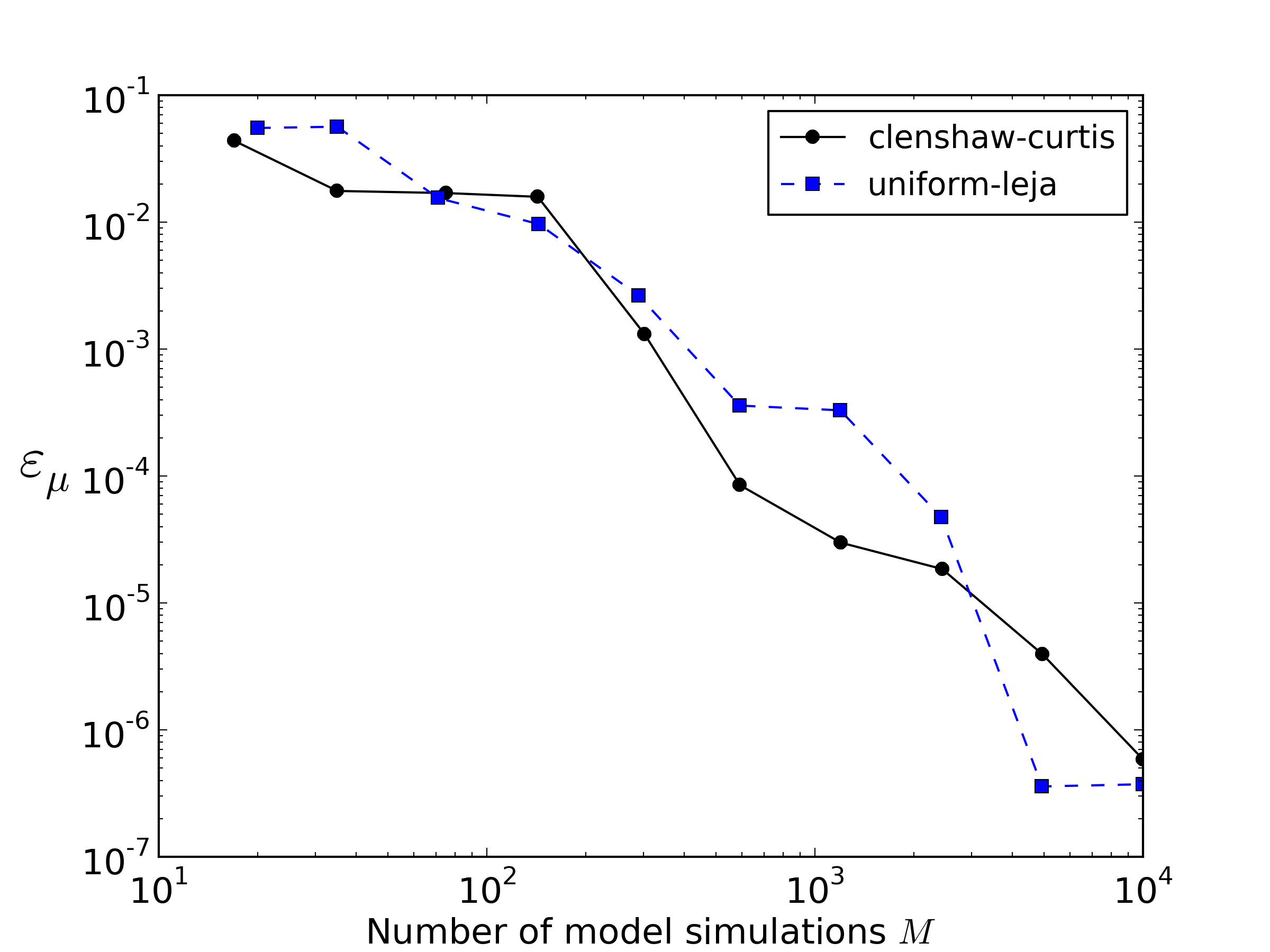}
\includegraphics[width=0.49\textwidth]{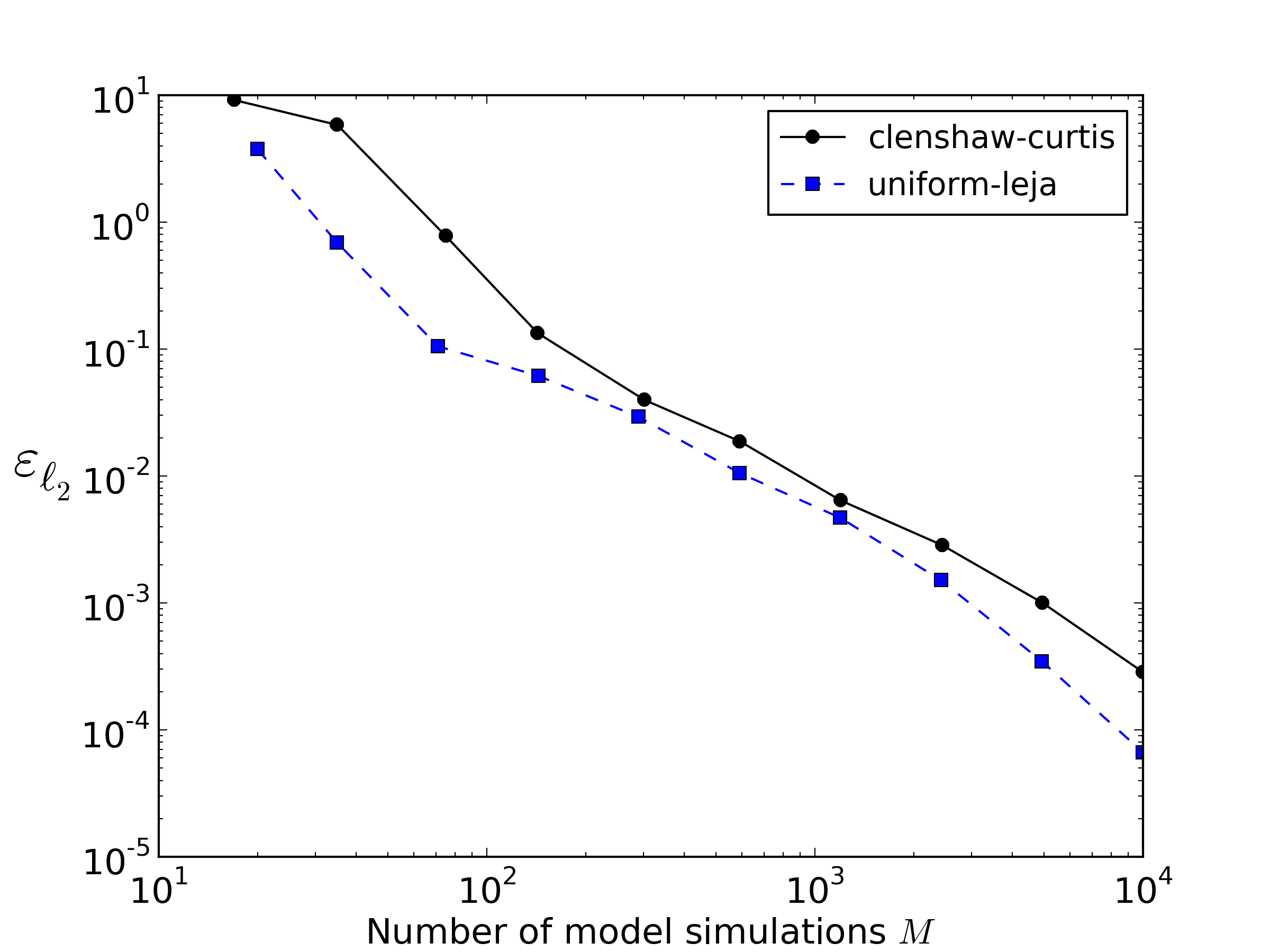}
\caption{Convergence of the mean and RMSE in the sparse grid approximation of the borehole model with respect to the number of model evaluations.}
\label{fig:borehole-d-40-rmse-convergence}
\end{figure}

\subsection{Heterogeneous diffusion equation}
In this section, we consider the heterogeneous diffusion equation in one-spatial 
dimension subject to uncertainty in the diffusivity coefficient. 
For $d \ge 1$ random dimensions:
\begin{equation}\label{eq:hetrogeneous-diffusion}
-\frac{d}{dx}\left[a(x,\bz)\frac{du}{dx}(x,\bz)\right]=1,\quad 
(x,\bz)\in(0,1)\times I_\bz
\end{equation}
subject to the physical boundary conditions
\begin{equation}
u(0)=0,\quad u(1)=0
\end{equation}
Furthermore assume that the random diffusivity satisfies
\begin{equation}\label{eq:diffusivityZ}
a(x,\brv)=\bar{a}+\sigma_a\sum_{k=1}^d\sqrt{\lambda_k}\phi_k(x)z_k
\end{equation} 
where $\{\lambda_k\}_{k=1}^d$ and $\{\phi_k(x)\}_{k=1}^d$ are, respectively, 
the eigenvalues and eigenfunctions of the covariance kernel 
\[
 C_a(x_1,x_2) = \exp\left[-\frac{(x_1-x_2)^2}{l_c^2}\right]
\]
The variability of the diffusivity field~\eqref{eq:diffusivityZ} is 
controlled by $\sigma_a$ and the correlation length
$l_c$ which determines the decay of the eigenvalues $\lambda_k$. Here we wish to approximate
the solution $u(1/3,\brv)$ when
$d=40$, $\sigma_a=0.021$ and $l_c=1/14$
and $z_k\in[-1,1]$, $k=1,\ldots,40$ to be independent and uniformly distributed random variables.
\begin{figure}[ht]
\centering
\includegraphics[width=0.49\textwidth]{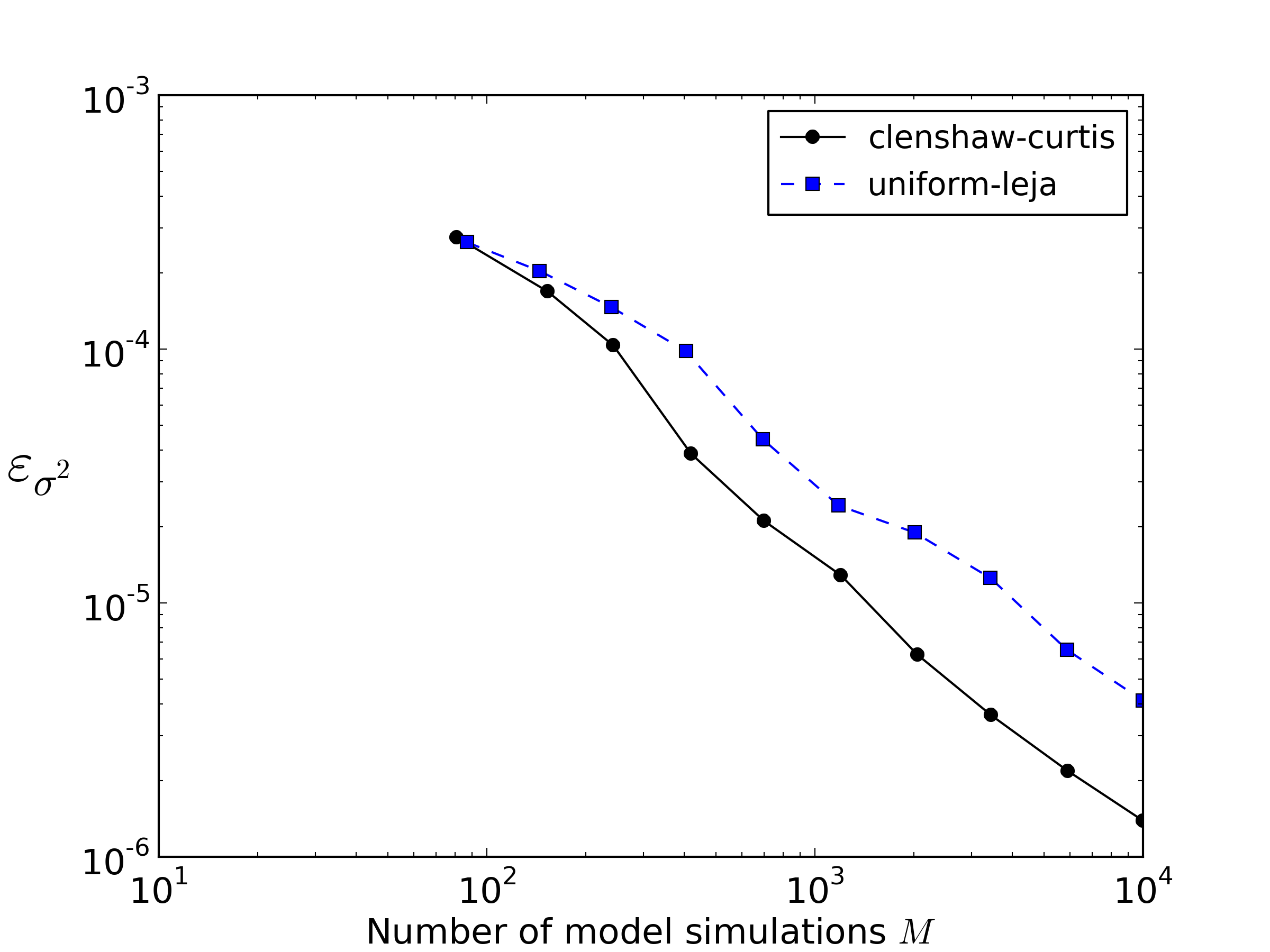}
\includegraphics[width=0.49\textwidth]{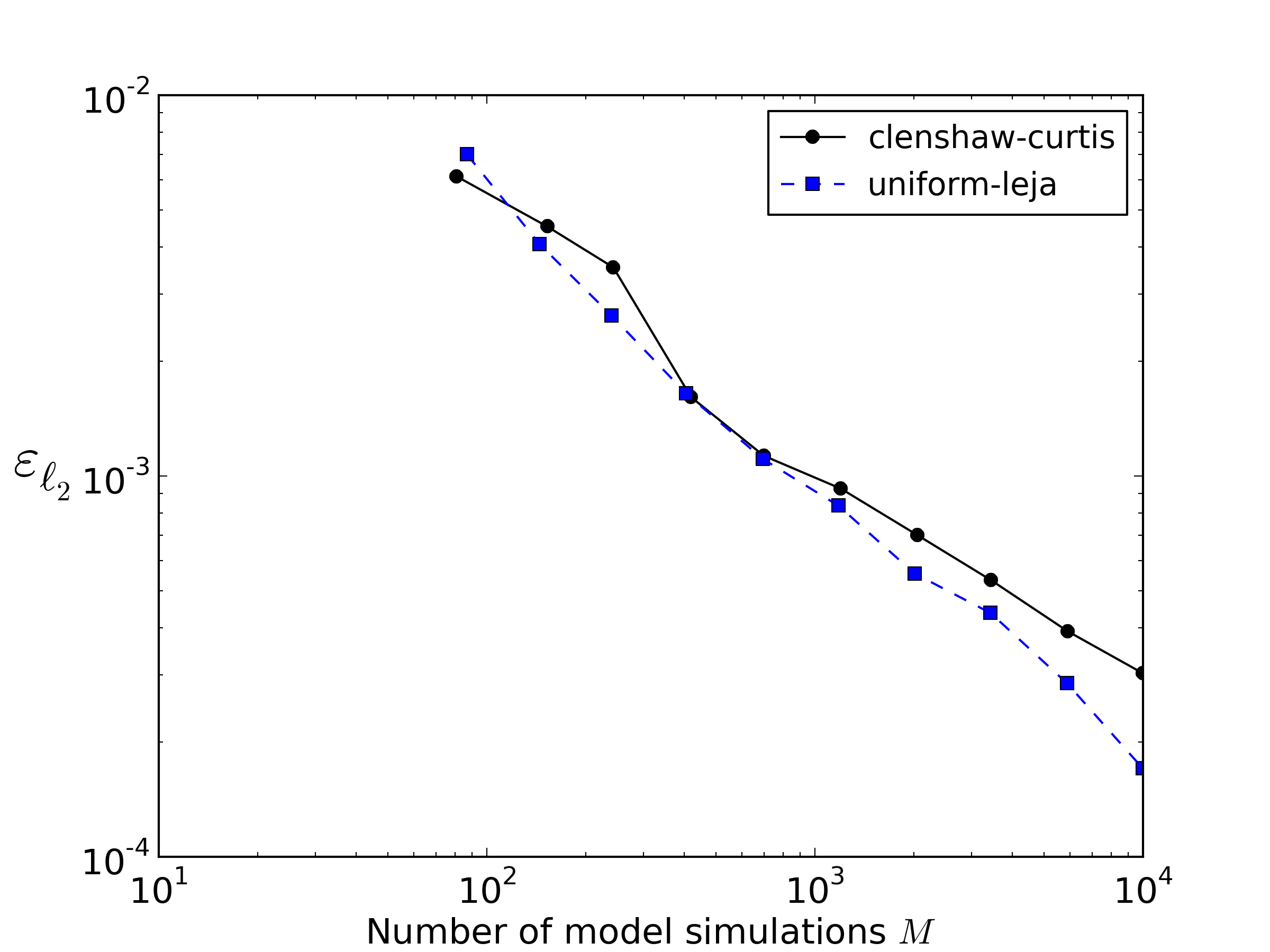}
\caption{Convergence of variance and RMSE in the sparse grid approximation of the solution to the diffusion equation with respect to the number of model evaluations} 
% for (left) $d=14$ and (right) $d=40$.}
\label{fig:steady-heat-rmse-convergence}
\end{figure}
We solve \eqref{eq:diffusivityZ} non-intrusively: for each node $\bxi_{\ell,\i}$ on our sparse grid, we use a finite-element discretization in $x$ to compute the solution. The comparison between Leja and Clenshaw-Curtis Smolyak construction is shown in Figure \ref{fig:steady-heat-rmse-convergence}. In this case, the Leja construction only performs marginally better than the CC approach for interpolation, and exhibits a now-familiar difficulty with quadrature. We explain this difference in the following way: for this equation, we certainly have dimensional anisotropy because the eigenvalues $\lambda_k$ decay. However, if we plot the parameter indices $\l$ for the subspaces $W_{\Bell}$ that contribute significantly to the solution, we will see an ellipsoid shape in index space. Thus, extra refinement performed by CC in certain directions is not wasted because these degrees of freedom can properly resolve mixed terms in parameter space. In this example, the granularity offered by Leja sequences is not needed or utilized.

\subsection{Resistor network}
\begin{figure}
\includegraphics[width=0.95\textwidth]{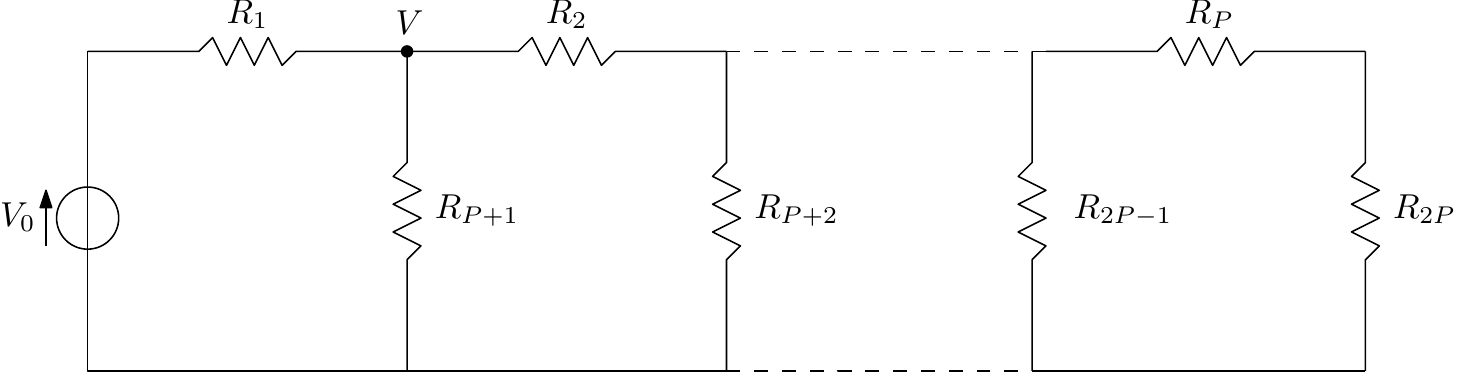}
\caption{Resistor network}\label{fig:resistor-network} 
\end{figure}

Consider the electrical resistor network shown in Figure~\ref{fig:resistor-network} \cite{roland}. 
The network is comprised of $d=2P$ resistances of uncertain
ohmage and the network is driven by a voltage source providing a
known potential $V_0$.  We are interested in determining how the voltage $V$ shown in 
Figure~\ref{fig:resistor-network} depends on the $d=2P$
resistances, which we take as random parameters 
%uniformly distributed in the interval $z_k\in[1-\varepsilon, 1+\varepsilon]$, $k=1,\ldots,d$.
% In this example we set $d=40$ and take the maximum perturbation to be $\varepsilon = 0.1$, 
% and set the reference potential $V_0 = 1$. 
that are independent and identically distributed Gaussian random variables with mean $\mu=1$ and standard deviation $\sigma=0.005$. Note for this value of $\sigma$, the probability that we encounter negative resistances is extremely small, and so apart from the obvious modeling error of possible negative resistances, no numerical difficulties are introduced. I.e. none of the sparse grid points or random samples used to generate the error resulted in a negative resistance. In this example we set $d=40$ and set the reference potential $V_0 = 1$. 

\begin{figure}[ht]
\begin{center}
  \includegraphics[width=0.49\textwidth]{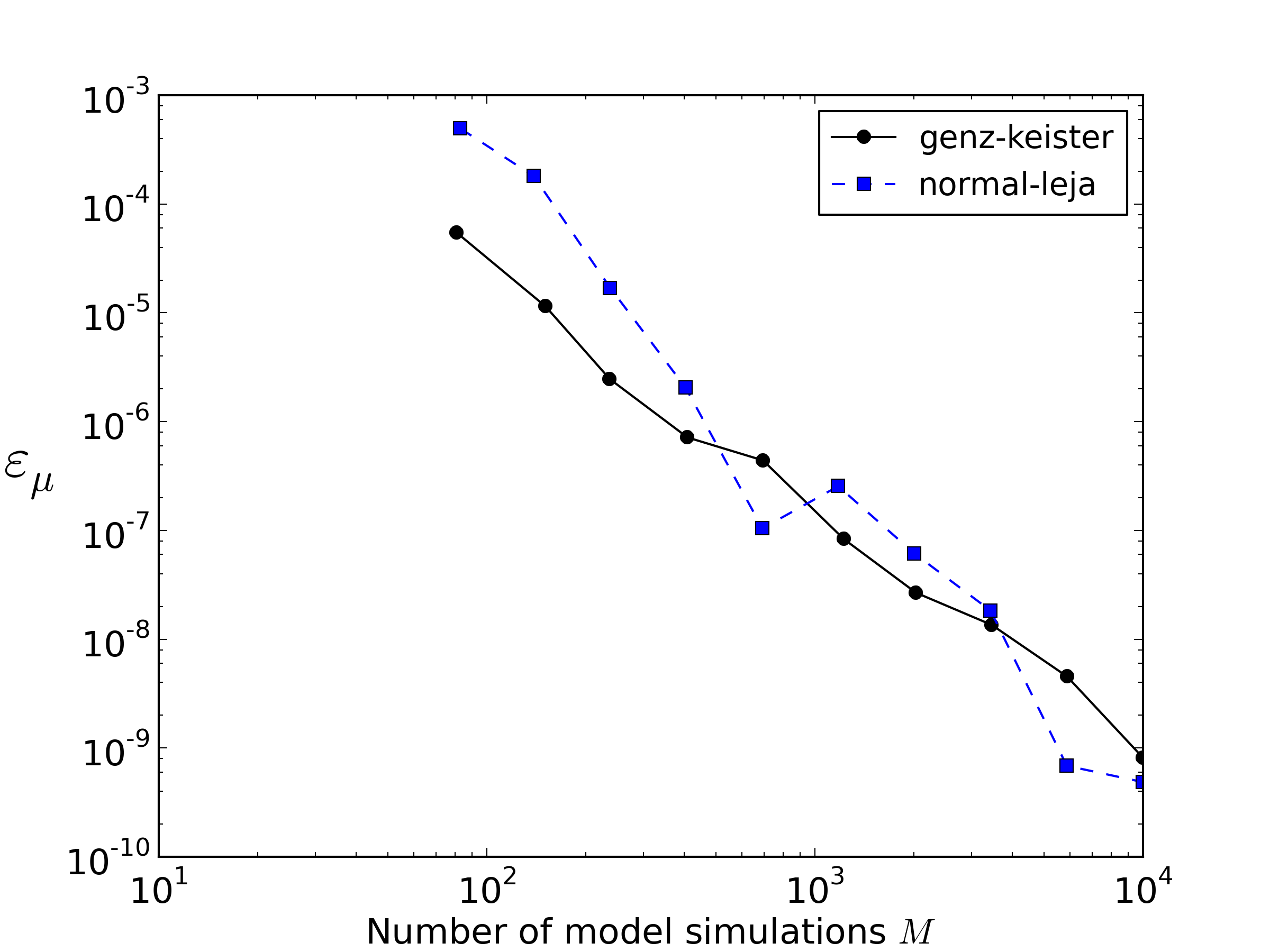}
  \includegraphics[width=0.49\textwidth]{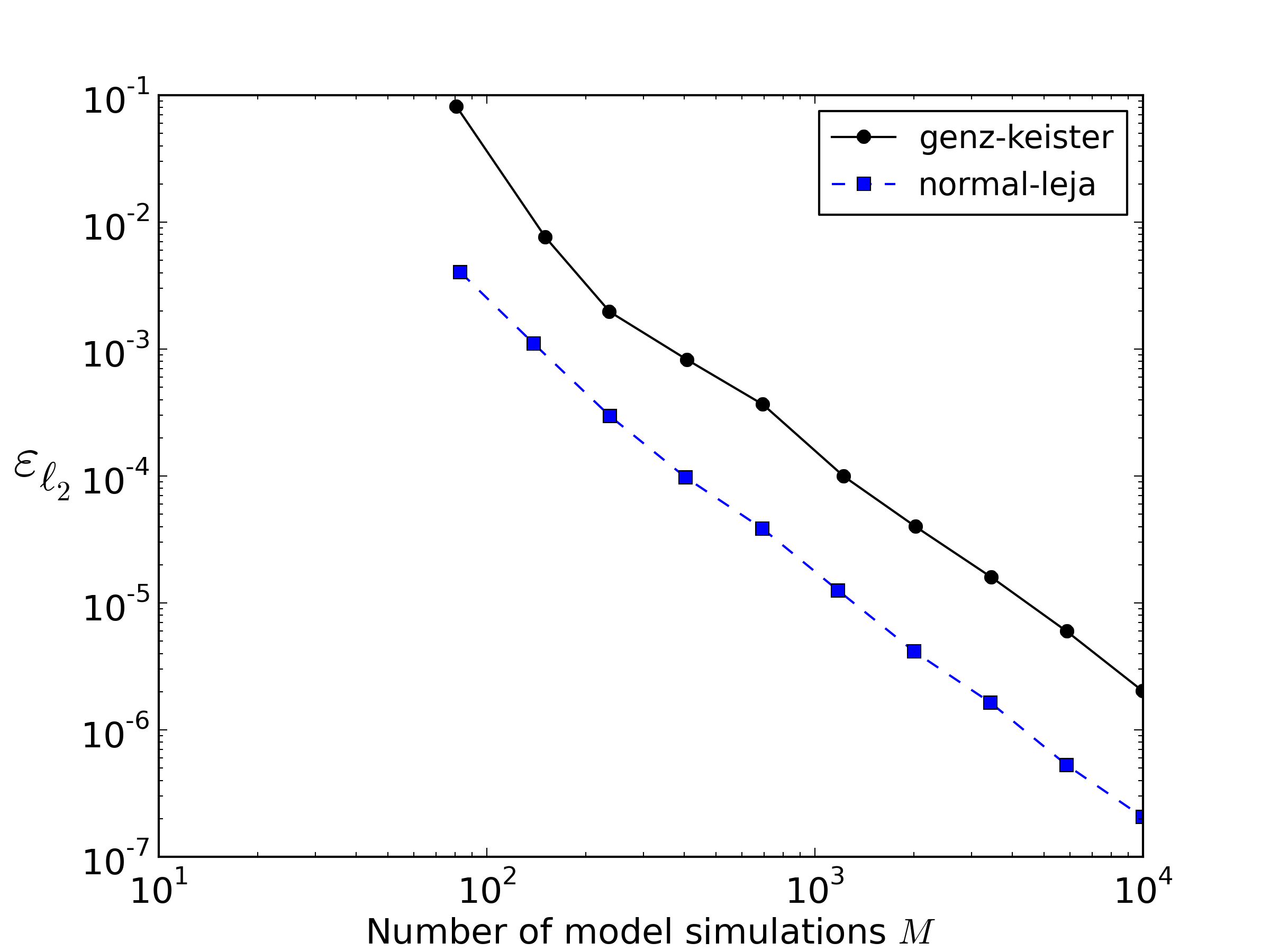}
\end{center}
\caption{Convergence of the Mean and RMSE in the sparse grid approximation of $V$ with respect to the number of model evaluations.}
\label{fig:resistor-d-40-rmse-convergence}
\end{figure}
In this case our comparison is not with CC, but with a nested Genz-Keister rule, which is one of the the standard ways to to peform nested interpolation and quadrature under a Gaussian weight \cite{genz_fully_1996}. Here our Leja rule is generated on an infinite domain with univariate weight function $w(z_j) = \exp(-(z_j - \mu)^2/(2 \sigma^2))$. In Figure \ref{fig:resistor-d-40-rmse-convergence} we see significant interpolatory improvement with the Leja rule, and even the quadrature results are competitive in this example.

}{
  
}

\section{Summary}\label{sec:summary}
\iftoggle{arxiv}{
  We have used Leja interpolatory grids as one-dimensional composite rules for adaptive Smolyak sparse grid construction. In one dimension, Leja rules are excellent interpolation grids and are a sequence, allowing one to generated nested rules with arbitrary granularity. We have shown that for several classical one-dimensional weight functions of interest, a corresponding weighted Leja rule produces a sequence whose empirical distribution asymptotically coincides with the limiting distribution for the Gauss quadrature nodes of the same family.

Using Leja rules to build up sparse grids in multiple dimensions grants the user a greater dexterity in adaptive refinement compared to more standard composite rules such as Clenshaw-Curtis. We have shown via several examples that the Leja rule can outperform standard high-order Smolyak constructions in interpolatory metrics, but are suboptimal when considering quadrature metrics. Design of Leja-like rules that are effective for approximating integrals will be the subject of future investigation.

}{
  
}

\section{Proof of Theorem \ref{thm:asymptotically-fekete}}\label{sec:proofs}
\iftoggle{arxiv}{
  \subsection{Weighted potential theory}

In this section we give the proof of Theorems \ref{thm:main-leja-result} and \ref{thm:asymptotically-fekete}. Indeed, if we show Theorem \ref{thm:asymptotically-fekete}, then well-established results imply Theorem \ref{thm:main-leja-result}. We recall our notation: $w$ is a given weight function associated to $z$. $v$ is the square root of $w$. $\widetilde{v}$ is related to $w$ through Theorem \ref{thm:main-leja-result}, and is essentially the square root of $w$, ignoring polynomial factors.

The proof of Theorem \ref{thm:asymptotically-fekete} relies on results from weighted potential theory. Potential theory is frequently explored in the complex plane $\C$, but we will restrict ourselves to subsets of $\C$ lying on the real axis. An excellent exposition of univariate weighted potential theory with comprehensive historical references is given in \cite{saff_logarithmic_1997}.

Let domain $I \subseteq \R$ with non-negative weight function $\widetilde{v}$ be given. On rather mild assumptions on $\widetilde{v}$ and $I$ then there exists a unique probability measure denoted $\mu_{\widetilde{v}}$ under which a weighted logarithmic energy for $Q = - \log \widetilde{v}$ is minimized:
\begin{align}\label{eq:logarithmic-energy}
  \mu_{\widetilde{v}} &= \argmin_{\mu\,:\, \mu(I) = 1} \int_I \int_I \log\frac{1}{|\eta - \xi| \widetilde{v}(\eta) \widetilde{v}(\xi)} \dx{\mu(\eta)}\dx{\mu(\xi)}  \\ \nonumber &= 
  \argmin_{\mu\,:\, \mu(I) = 1} \left[ \int_I \int_I \log\frac{1}{|\eta - \xi|} \dx{\mu(\eta)}\dx{\mu(\xi)} + 2 \int_I Q(\xi) \dx{\mu(\xi)}\right]
\end{align}
A common physical analogy of the above is to find the minimum-energy electrostatic charge distribution (measure) in a region $I$ when an external electrostatic field $Q$ is applied. The measure $\mu_{\widetilde{v}}$ is the weighted equilibrium measure of $I$ in the presence of the field $Q$. Even if $I$ is unbounded, $\mu_{\widetilde{v}}$ always has compact support. 

A discrete version of the above optimization problem is furnished by the concept of Fekete points. An array of points which maximizes a weighted Vandermonde matrix determinant is called a set of Fekete points; this weighted determinant is a discrete, unnormalized Monte-Carlo-like estimate of the negative exponential of the integral in \eqref{eq:logarithmic-energy}. A set of points $\xi_0, \ldots, \xi_n$ is an array of $\widetilde{v}$-weighted Fekete points if it satisfies
\begin{align}\label{eq:weighted-fekete-objective}
  \{\xi_0, \ldots \xi_n\} = \argmax_{z_0, \ldots, z_n} V(z_0, \ldots, z_n) \prod_{j=0}^n \widetilde{v}^n(z_j) = \argmax_{z_0, \ldots, z_n} \prod_{0 \leq j < k \leq n} \left| z_j - z_k \right| \widetilde{v}(z_j) \widetilde{v}(z_k),
\end{align}
where $V(z_0, \ldots, z_n)$ is modulus determinant of the $(n+1)\times(n+1)$ Vandermonde matrix $W$ with entries $W_{r,s} = z_r^s$ for $0 \leq r, s \leq n$. Fekete sets are not necessarily unique, and are notoriously difficult to compute exactly. However, Fekete points are excellent interpolation/approximation nodal sets. The $n\rightarrow \infty$ limiting behavior of the weighted Vandermonde determinant for Fekete nodes is described by the weighted transfinite diameter $\delta_{\widetilde{v}}$ ($=$ weighted logarithmic capacity \cite{mhaskar_weighted_1992}). Let $V_n$ denote the maximum weighted determinant from \eqref{eq:weighted-fekete-objective} achieved by Fekete points. Then
\begin{align}\label{eq:transfinite-diameter}
  \lim_{n\rightarrow\infty} \left(V_n\right)^{1/m_n} = \delta_{\widetilde{v}},
\end{align}
where $m_n \triangleq 1 + 2 + \cdots + n = \frac{n(n+1)}{2}$. Any triangular array of nodes whose determinant behaves like \eqref{eq:transfinite-diameter} is called \textit{asymptotically weighted Fekete}, alluding to the fact that the determinant is not exactly maximum, but is asymptotically comparable to Fekete points.

The connection between Fekete nodes and the equilibrium measure is established by the following result: if a triangular array of nodes $\left\{\xi_{j,n}\right\}_{j \leq n}$ is asymptotically weighted Fekete, then its empirical measure distributionally converges to the weighted equilibrium measure:
\begin{lemma}[\cite{totik_fast_1994,saff_logarithmic_1997}]\label{lemma:fekete-limit}
  Suppose $\xi_{j,n}$ is a triangular array of points satisfying
  \begin{align*}
    \lim_{n \rightarrow \infty} \left[ V\left(\xi_{0,n}, \ldots, \xi_{n,n} \right) \prod_{j=0}^n \widetilde{v}\left(\xi_{j,n} \right) \right]^{1/m_n} = \delta_{\widetilde{v}}.
  \end{align*}
  Then
  \begin{align*}
    \lim_{n\rightarrow\infty} \frac{1}{n} \sum_{j=1}^n \delta_{\xi_{j,n}} = \mu_{\widetilde{v}},
  \end{align*}
  where $\delta_z$ is the Dirac mass centered at $z$.
\end{lemma}
Here (and elsewhere when discussing convergence of measures) equality is in the weak-$\ast$ sense. Our strategy is to first show Theorem \ref{thm:asymptotically-fekete}, that a contracted version of the $w$-weighted Leja points from \eqref{eq:weighted-leja-objective} are $\widetilde{v}$-weighted asymptotically Fekete. This will allow us to immediately use Lemma \ref{lemma:fekete-limit} to conclude Theorem \ref{thm:main-leja-result}. 

To proceed, we will need the notion of Chebyshev constants. Given a potential-theoretic admissible weight $\widetilde{v}$ on $I$, the weighted Chebyshev constant of order $n$ is
\begin{align*}
  \tau^{(n)}_{\widetilde{v}} = \inf \left\{ \left\| \widetilde{v}^n(z) p_n(z) \right\|_I\;\; | \;\; p_n(z) = z^n + q_{n-1}(z)\;\; \forall\;\; q_{n-1} \in P^{n-1}\right\},
\end{align*}
where $\|\cdot\|_I$ is the sup-norm on the domain $I$, and $P^n$ is the space of polynomials of degree $n$ or less. The sequence of constants $\left\{ \left( \tau^{(n)}_{\widetilde{v}}\right)^{1/n} \right\}$ is a decreasing sequence with a limit:
\begin{align*}
  \inf_{n} \left( \tau^{(n)}_{\widetilde{v}} \right)^{1/n} = \lim_{n\rightarrow \infty} \left( \tau^{(n)}_{\widetilde{v}} \right)^{1/n} \triangleq \tau_{\widetilde{v}}
\end{align*}
This limit is called the weighted Chebyshev constant \cite{mhaskar_weighted_1992}, and coincides with the transfinite diameter $\delta_{\widetilde{v}}$ in the unweighted case, but is distinct in general weighted scenarios. Note that by definition, 
\begin{align}\label{eq:chebyshev-constant-bound}
  \left\| \widetilde{v}^n(z) ( z^n + q_{n-1}(z) )\right\|_I \geq \left(\tau_{\widetilde{v}}\right)^n.
\end{align}
for any $q_{n-1} \in P^{n-1}$. In general, the relation between the Chebyshev constant $\tau_{\widetilde{v}}$ and the transfinite diameter $\delta_{\widetilde{v}}$ is given by
\begin{align}\label{eq:chebyshev-transfinite-relation}
  \tau_{\widetilde{v}} &= \delta_{\widetilde{v}} \exp\left( \int_S Q(z)\, \dx{\mu_{\widetilde{v}}(z)}\right), & S&= \textrm{supp}\, \mu_{\widetilde{v}}.
\end{align}

For the Jacobi, Hermite, and Laguerre cases mentioned in Theorem \ref{thm:main-leja-result}, our goal is to prove the result \eqref{eq:leja-asymptotically-fekete}. Our contraction proofs are different for the cases of bounded $I$ versus unbounded $I$. We first consider the bounded Jacobi case. 

\subsection{Jacobi case}\label{sec:proof-jacobi}
We prove that the Leja maximization scheme \eqref{eq:weighted-leja-objective} produces nodes whose limiting distribution is the \textit{unweighted} equilibrium measure, or the arcsine measure. For the unweighted case we have $\widetilde{v} \equiv 1$, the contraction is $k_n \equiv 1$ (and so is omitted), and the transfinite diameter $\delta$ and the Chebyshev constant $\tau$ are identical:
\begin{align}\label{eq:delta-tau-unweighted}
  \delta = \tau
\end{align}
We begin by considering the proof assuming $\alpha, \beta \geq 0$. We have $v(z) = f^{(\alpha/2,\beta/2)}(z) = (1-z)^{\alpha/2} (1+z)^{\beta/2}$ for $\alpha, \beta \geq 0$ over $I = [-1,1]$. We will need the following constants:
\begin{align*}
  C_1 &= C_1(\alpha, \beta) \triangleq \left\| f^{(\alpha/2, \beta/2)} \right\|_\infty < \infty \\
  C_2 &= C_2(\alpha, \beta) \triangleq \left\| \frac{f^{(\lceil\alpha/2\rceil, \lceil \beta/2 \rceil)}(z)}{f^{(\alpha/2,\beta/2)}(z)} \right\|_\infty < \infty
\end{align*}
With these constants, we have
\begin{align*}
  1 \geq \frac{v(z)}{C_1} = \frac{f^{(\alpha/2,\beta/2)}(z)}{C_1} \geq \frac{f^{(\lceil\alpha/2\rceil, \lceil \beta/2 \rceil)}(z)}{C_1 C_2}, \\
\end{align*}
We now note that $\lceil \alpha/2 \rceil$ and $\lceil \beta/2 \rceil$ are integers, and for shorthand we write $\gamma = \lceil \alpha/2 \rceil + \lceil \beta/2 \rceil$. Then the right-hand side of the above equation is a monic polynomial (modulo sign) of degree $\gamma$. Therefore, we have 
\begin{align*}
  V(z_0, \ldots, z_n) = \prod_{j=1}^n \prod_{k=0}^{j-1} |z_j - z_k| &\geq 
  \frac{1}{C_1^n} \prod_{j=1}^n v(z_j) \prod_{k=0}^{j-1} |z_j - z_k| \\
  &= \frac{1}{C_1^n} \prod_{j=1}^n \left\| v(z) \prod_{k=0}^{j-1} |z - z_k| \right\|_I \\
  &\geq \frac{1}{C_1^n C_2^n} \prod_{j=1}^n \left\| f^{(\lceil \alpha/2 \rceil, \lceil \beta/2 \rceil)}(z) \prod_{k=0}^{j-1} |z - z_k| \right\|_I \\
  &\geq \frac{1}{C_1^n C_2^n} \prod_{j=1}^n \tau^{j+\gamma} = \frac{\tau^{m_n} \tau^{n \gamma}}{C_1^n C_2^n}
\end{align*}
Thus, we have 
\begin{align*}
  \tau^{m_n} \left[\frac{\tau^{\gamma}}{C_1 C_2}\right]^n \leq V(z_0, \ldots, z_n) \leq \max_{x_0, \ldots x_n \in \Xi} V(x_0, \ldots, x_n)
\end{align*}
We raise all the above to the $1/m_n$ power, which yields
\begin{align*}
  \lim_{n\rightarrow\infty} \left[ V(z_0, \ldots, z_n) \right]^{1/m_n} \geq \tau &= \delta \\
  \lim_{n\rightarrow\infty} \left[ V(z_0, \ldots, z_n) \right]^{1/m_n} \leq & \delta
\end{align*}
We have thus proven that Jacobi-weighted Leja sequences for $\alpha, \beta \geq 0$ are $\widetilde{v}$-asymptotically Fekete, i.e., we have proven \eqref{eq:leja-asymptotically-fekete}. 

For $-1 < \alpha, \beta < 0$, we proceed without loss under the assumption that both $\alpha$ and $\beta$ are negative. In this case then $\xi_0 \neq \pm 1$ is arbitrarily chosen so that
\begin{align*}
  z_1 = \argmax_z v(z) | z - z_0 |.
\end{align*}
But since $\widetilde{v}$ is infinite at the endpoints $\pm 1$, then the maximum is achieved at one of these points, say $z_1 = +1$. Then the Leja iteration continues:
\begin{align*}
  z_2 = \argmax_z v(z) |z - z_0| |z - z_1|.
\end{align*}
Now since $z_1 = +1$, then the last term is identical to $(1 - z)$, and we may absorb this term into the weight, giving a more explicit formula:
\begin{align*}
  z_2 = \argmax_z (1-z)^{\alpha/2+1} (1+z)^{\beta/2} |z - z_0|.
\end{align*}
But since $\alpha/2 +1 > 0$, then the maximum (infinity) is now achieved at $z = -1$, which becomes $z_2$. Now we choose $z_3$ and again combine the terms $|z - z_1|$ and $|z - z_2|$ into the weight function:
\begin{align*}
  z_3 = \argmax_z (1-z)^{\alpha/2 +1} (1+z)^{\alpha/2 + 1} \prod_{k=0 \atop k\neq 1, 2}^{2} |z - z_k|
\end{align*}
Proceeding in this way, future Leja points are chosen according to
\begin{align*}
  z_{n} &= \argmax_z f^{(\alpha/2+1, \beta/2+1)}(z) \left( \prod_{k=0 \atop k \neq 1, 2}^{n-1} |z - z_k|\right) 
          %&= \argmax_z f^{(\alpha/2+1,\beta/2+1)}(z) \prod_{k=0 \atop k\neq 1,2}^{n-1} |z - z_k|
\end{align*}
In other words, we choose $z_n$ as a Leja optimization with a new weight function $f^{(\alpha/2+1, \beta/2+1)}$, whose parameters are $\alpha/2+1 > 0$ and $\beta/2+1 >0$. This effectively reduces the problem to the case where $\alpha, \beta > 0$. Then as before, this Leja sequence is asymptotically (unweighted) Fekete and so has empirical measure that converges to the arcsine measure $\mu$.

\subsection{Hermite case}\label{sec:proof-hermite}
In this section, $I = \R$ (and for shorthand write $\|\cdot\|_I = \| \cdot \|$) and $w(z) = z^{2 \mu} \exp(-2 |z|^{\alpha})$. (Compared to Theorem \ref{thm:main-leja-result}, in this section we have redefined $z \gets 2^{1/\alpha} z$ to make the computations cleaner.) The contraction factor is $k_n = n^{-1/\alpha}$. As usual, define $v(z) = \sqrt{w(z)}$. 
%With $\mu = 0$, 
The limit weight $\widetilde{v}$ from Theorem \ref{thm:main-leja-result} is $\widetilde{v}(z) = \exp(-|z|^\alpha)$. For this $\widetilde{v}$, it is known that the weighted Chebyshev constant is related to the weighted transfinite diameter by the relation
\begin{align}\label{eq:hermite-chebyshev-transfinite-relation}
  \exp\left(-\frac{1}{2\alpha}\right) \tau_{\widetilde{v}} = \delta_{\widetilde{v}}.
\end{align}
See e.g., \cite{mhaskar_extremal_1984,mhaskar_weighted_1992}. We choose a weighted Leja sequence of points $z_n$ according to \eqref{eq:weighted-leja-objective}.
%, and in terms of $v$ this is
%\begin{align}\label{eq:leja-maximization}
%  z_n = \argmax_{z} v(z) \prod_{j=0}^{n-1} \left| z - z_j \right|,
%\end{align}
%where $z_0$ is arbitrarily chosen. 
The negative log-weight of $\widetilde{v}$ is $Q = |z|^\alpha$, and it is a homogeneous function with homogeneity exponent $\alpha$:
\begin{align*}
  Q(cz) &= c^\alpha Q(z), & \forall\;\; c > 0, z \in \C
\end{align*}
Then the following are easily verified: for any $n > 0$:
\begin{align}\label{eq:sqrt-weight-homogeneity}
  \widetilde{v}\left( n^{1/\alpha} z\right) &= \widetilde{v}^n(z), & \widetilde{v}^n\left( n^{-1/\alpha} z\right) = \widetilde{v}(z)
\end{align}
Given the Leja sequence $z_n$ from \eqref{eq:weighted-leja-objective}, then for each $n$ we define progressively contracted versions of the grid:
\begin{align*}
  z_{j,n} &= n^{-1/\alpha} z_j, & j, n = 1, \ldots
\end{align*}
We will only need $z_{j,n}$ for $j \leq n$: we view $z_{j,n}$ as a triangular array of points with $j \leq n$.  We note that one can easily transform between one contracted set of nodes and another:
\begin{align*}
  z_{j,n} = \left(\frac{n-1}{n} \right)^{1/\alpha} z_{j,n-1},
\end{align*}
which in turn implies that for our particular choice of family of weight functions:
\begin{align*}
  \widetilde{v}^n\left(z_{j,n}\right) = \widetilde{v}^{n-1} \left(z_{j,n-1}\right) = \cdots = \widetilde{v}\left(z_{j,1}\right) = \widetilde{v}\left(z_j\right)
\end{align*}
We wish to prove that the array $z_{j,n}$ is $\widetilde{v}$-weighted asymptotically Fekete. I.e., we wish to prove
\begin{align}\label{eq:weighted-fekete-condition}
  \lim_{n \rightarrow \infty} \left[ \left| \det V\left(z_{0,n}, \ldots, z_{n,n}\right) \right| \prod_{j=0}^n \widetilde{v}^n\left(z_{j,n}\right) \right]^{1/m_n} = \delta_{\widetilde{v}},
\end{align}

We first consider the case with the parameter $\mu > 0$. 
%As in Theorem \ref{thm:main-leja-result}, the asymptotic weight measure $\widetilde{v}$ is now distinct from $v$: $\widetilde{v} = \exp(-|z|^{\alpha})$. For a Leja sequence $z_j$, we still have \eqref{eq:sqrt-weight-homogeneity} but for the weight $\widetilde{v}$. We wish to prove
Using the explicit Vandermonde determinant formula \eqref{eq:weighted-fekete-objective} yields the following formula that we wish to prove:
\begin{align}\label{eq:hermite-wish}
  \lim_{n \rightarrow \infty} \left[ \widetilde{v}^n(z_{0,n}) \prod_{j=1}^n \widetilde{v}^n(z_{j,n}) \prod_{k=0}^{j-1} \left| z_{j,n} - z_{k,n} \right|\right]^{1/m_n} = \delta_{\widetilde{v}}.
\end{align}
As with the Jacobi case, showing that the limit is $\leq \delta_{\widetilde{v}}$ is straightforward from the definition of $\delta_{\widetilde{v}}$, so we concentrate on the inequality $\geq$. We have 
\begin{align}\label{eq:hermite-break}
  \widetilde{v}^n(z_{0,n}) \prod_{j=1}^n \widetilde{v}^n(z_{j,n}) \prod_{k=0}^{j-1} \left| z_{j,n} - z_{k,n} \right| &\stackrel{\eqref{eq:sqrt-weight-homogeneity}}{=} \widetilde{v}(z_0) \prod_{j=1}^n n^{-j/\alpha} \underbrace{\widetilde{v}(z_j) \prod_{k=0}^{j-1} \left| z_j - z_k \right|}_{\text{(a)}}.
\end{align}
For the term (a), we note that $z_j \neq 0$ for $j \geq 1$ since $\mu > 0$. Therefore, we may write this term as
\begin{align*}
  \textrm{(a)} &= \underbrace{|z_j|^{-\mu}}_{\text{(aa)}} \underbrace{v(z_j) \prod_{k=0}^{j-1} \left| z_j - z_k \right|}_{\text{(ab)}}.
\end{align*}
To compute lower bounds for terms (aa) and (ab) we will need the following notation: for the $\widetilde{v}$-weighted equilibrium measure, we have
\begin{align}\label{eq:mu-support}
  S = \textrm{supp}\, \mu_{\widetilde{v}} &= [-c,c], & S_n &= \left[ -n^{1/\alpha} c, n^{1/\alpha} c\right].
\end{align}
The constant $c$ is $2^{1/\alpha} b(\alpha)$, with $b(\alpha)$ given in the ``Hermite" case of Table \ref{tab:summary-table}, or in Table \ref{tab:hermite-table}.

To bound (aa), we know that $z_j$ was computed from the optimization \eqref{eq:weighted-leja-objective}. This allows us to derive an upper bound for the magnitude of these weighted Leja points.
\begin{lemma}
  Let $M = \lceil \mu \rceil$. For each $j$, we have the following bound for term (aa):
  \begin{align}\label{eq:zj-bound}
    |z_j|^{-\mu} \geq (M+j)^{-\mu / \alpha} c^{-\mu}
  \end{align}
\end{lemma}
\begin{proof}
  We make use of the following result \cite{mhaskar_extremal_1984} that compactifies the set on which the supremum norm of a weighted polynomial ``lives" for our exponential weights:
\begin{subequations}
\begin{align}
  \label{eq:supnorm-house-1}
  \left\| \widetilde{v}(z) \left( z^n + q_{n-1}\right) \right\|_\R &= 
  \left\| \widetilde{v}(z) \left( z^n + q_{n-1}\right) \right\|_{S_n},
  & \forall \,\, q_{n-1} \in P_{n-1} \\
  \label{eq:supnorm-house-2}
  \left\| \widetilde{v}^n(z) \left( z^n + q_{n-1}\right) \right\|_\R &= 
  \left\| \widetilde{v}^n(z) \left( z^n + q_{n-1}\right) \right\|_{S},
  & \forall \,\, q_{n-1} \in P_{n-1}
\end{align}
\end{subequations}
where $S$ and $S_n$ are given by \eqref{eq:mu-support}. Note that $z_j$ is chosen as in \eqref{eq:weighted-leja-objective} with
\begin{align*}
v(z_j) \prod_{k=0}^{j-1} \left| z_j - z_k \right|
   = \left\| v(z) \prod_{k=0}^{j-1} \left| z - z_k \right| \right\|
   = \left\| z^\mu \widetilde{v}(z) \prod_{k=0}^{j-1} \left| z - z_k \right| \right\|
\end{align*}
If $\mu$ is an integer, then the argument under the norm is a $\widetilde{v}$-weighted monic polynomial of degree $\mu + j$. Thus, the extremum of the argument is achieved on the set $S_{\mu+j}$, which implies that the smallest-magnitude maximizer as stipulated in \eqref{eq:maximizer-choice} satisfies $z_j \in S_{\mu + j}$. The result \eqref{eq:zj-bound} follows.

If $\mu$ is not an integer, then consider the function
\begin{align*}
  f(z) = \left|z\right|^M \widetilde{v}(z) \prod_{k=0}^{j-1} \left| z - z_k \right|
%  &= \left\| z^M \widetilde{v}(z) \prod_{k=0}^{j-1} \left| z - z_k \right| \right\|  \\
%  &= \left\| z^M \widetilde{v}(z) \prod_{k=0}^{j-1} \left| z - z_k \right| \right\|_{A_{M+j}}
\end{align*}
We know that $z_j$ is a maxmizer of $\left|z\right|^{\mu - M} f(z)$. Let $z_\ast$ be the smallest-magnitude maximizer of $f(z)$. Suppose $\left|z_j\right| > \left|z_\ast\right|$; since $\mu - M < 0$ we have
\begin{align*}
  \left| z_j \right|^{\mu - M} f(z_j) < \left| z_\ast \right|^{\mu - M} f(z_j) < \left| z_\ast \right|^{\mu - M} f(z_\ast),
\end{align*}
which is a contradiction since $z_j$ maximizes the norm of $z^{\mu - M} f(z)$. Therefore, $\left|z_j\right| \leq \left|z_\ast\right|$. But $z_\ast$ maximizes $f$, which is a $\widetilde{v}$-weighted polynomial of degree $M+j$. Therefore, by \eqref{eq:supnorm-house-1} we have $z_\ast \in S_{M + j}$. This in turn implies $z_j \in S_{M+j}$, and again \eqref{eq:zj-bound} follows.
\end{proof}

%Now, if $|z_j| < 1$, then trivially we have $|z_j|^{-\mu} \geq 1$. If instead $|z_j| > 1$, then 
%\begin{align*}
%  v(z_j) \prod_{k=0}^{j-1} \left|z_j - z_k \right| &= \left| z_j\right|^{\mu} \widetilde{v}(z_j) \prod_{k=0}^{j-1} \left|z_j - z_k \right| \\
%                                                   & \leq \left| z_j \right|^M \widetilde{v}(z_j) \prod_{k=0}^{j-1} \left|z_j - z_k \right| \\
%\end{align*}
%The last 

We likewise have a bound for the (ab) term:
\begin{lemma}
  For every $j$, we have the following bound for term (ab):
  \begin{align}\label{eq:ab-bound}
    v(z_j) \prod_{k=0}^{j-1} \left| z_j - z_k \right| \geq \frac{j^{(j+\mu)/\alpha}}{c^{M-\mu}} \left( \tau_{\widetilde{v}}\right)^{j+M},
  \end{align}
  where $c$ is the constant from \eqref{eq:mu-support}.
\end{lemma}
\begin{proof}
  Since, by \eqref{eq:weighted-leja-objective}, $z_j$ maximizes the left-hand side of \eqref{eq:ab-bound}, we are concerned with bounding
\begin{align*}
  \left\| v(z) \prod_{k=0}^{j-1} \left| z - z_k \right|\right\| = 
  \left\| z^\mu \widetilde{v}(z) \prod_{k=0}^{j-1} \left| z - z_k \right|\right\|.
\end{align*}
We contract $z \gets j^{-1/\alpha} z$ to obtain
\begin{subequations}
\begin{align}\label{eq:lemma2-sub1}
  \left\| z^\mu \widetilde{v}(z) \prod_{k=0}^{j-1} \left| z - z_k \right|\right\| = 
  j^{(j+\mu)/\alpha} \left\| z^\mu \widetilde{v}^j(z) \prod_{k=0}^{j-1} \left| z - z_{k,j} \right|\right\|.
\end{align}
We also have the following properties:
\begin{align}\label{eq:lemma2-sub2}
  \min_{z \in S} z^{\mu - M} &= c^{\mu - M} \\
  \label{eq:lemma2-sub3}
  \left\|\widetilde{v}\right\| &= 1
\end{align}
\end{subequations}
 Thus we have
\begin{align*}
  v(z_j) \prod_{k=0}^{j-1} \left| z_j - z_k \right| &\stackrel{\eqref{eq:weighted-leja-objective}}{=} 
  \left\| v(z) \prod_{k=0}^{j-1} \left| z - z_k \right| \right\| \\
  &\stackrel{\eqref{eq:lemma2-sub1}}{=} 
  j^{(j+\mu)/\alpha} \left\| z^\mu \widetilde{v}^j(z) \prod_{k=0}^{j-1} \left| z - z_{k,j} \right|\right\| \\
  %\left\| z^\mu \widetilde{v}^j(z) \prod_{k=0}^{j-1} \left| z - z_{k,j} \right|\right\| &\geq 
  &\geq
  j^{(j+\mu)/\alpha} \left\| z^\mu \widetilde{v}^j(z) \prod_{k=0}^{j-1} \left| z - z_{k,j} \right|\right\|_S \\
  &\stackrel{\eqref{eq:lemma2-sub2}}{\geq}
  j^{(j+\mu)/\alpha} \frac{1}{c^{M-\mu}} \left\| z^M \widetilde{v}^j(z) \prod_{k=0}^{j-1} \left| z - z_{k,j} \right|\right\|_S \\
  &\stackrel{\eqref{eq:lemma2-sub3}}{\geq}
  j^{(j+\mu)/\alpha} \frac{1}{c^{M-\mu}} \left\| z^M \widetilde{v}^{j+M}(z) \prod_{k=0}^{j-1} \left| z - z_{k,j} \right|\right\|_S \\
  &\stackrel{\eqref{eq:supnorm-house-2}}{=}
  j^{(j+\mu)/\alpha} \frac{1}{c^{M-\mu}} \left\| z^M \widetilde{v}^{j+M}(z) \prod_{k=0}^{j-1} \left| z - z_{k,j} \right|\right\|_{\R} \\
  &\stackrel{\eqref{eq:chebyshev-constant-bound}}{\geq}
  j^{(j+\mu)/\alpha} \frac{1}{c^{M-\mu}} \left( \tau_{\widetilde{v}}\right)^{j+M}.
\end{align*}
\end{proof}
With these two lemmas obtained, we can bound term (a) from \eqref{eq:hermite-break}:
\begin{align*}
  \textrm{(a)} &= |z_j|^{-\mu} v(z_j) \prod_{k=0}^{j-1} \left| z_j - z_k \right| \\
               &\stackrel{\eqref{eq:zj-bound}}{\geq} \frac{1}{(M+j)^{\mu/\alpha} c^{\mu}} v(z_j) \prod_{k=0}^{j-1} \left| z_j - z_k \right| \\
               &\stackrel{\eqref{eq:ab-bound}}{\geq} 
               \frac{j^{(j+\mu)/\alpha}}{(M+j)^{\mu/\alpha} c^{M}} \left( \tau_{\widetilde{v}}\right)^{j+M}
\end{align*}
We can therefore bound the entire weighted determinant from \eqref{eq:hermite-break}:
\begin{align*}
  \widetilde{v}^n(z_{0,n}) \prod_{j=1}^n \widetilde{v}^n(z_{j,n}) \prod_{k=0}^{j-1} \left| z_{j,n} - z_{k,n} \right| &= \widetilde{v}(z_0) \prod_{j=1}^n n^{-j/\alpha} \underbrace{\widetilde{v}(z_j) \prod_{k=0}^{j-1} \left| z_j - z_k \right|}_{\text{(a)}} \\
  &\geq \widetilde{v}(z_0) \prod_{j=1}^n n^{-j/\alpha}
               \frac{j^{(j+\mu)/\alpha}}{(M+j)^{\mu/\alpha}c^{M}} \left( \tau_{\widetilde{v}}\right)^{j+M} \\
               &= \widetilde{v}(z_0) \left[ \prod_{j=1}^n \frac{j}{(M+j)} \right]^{\mu/\alpha} \left[ \frac{\tau_{\widetilde{v}}}{c}\right]^{M n} \left[ \prod_{j=1}^n \left(\frac{j}{n}\right)^j \right]^{1/\alpha} \left( \tau_{\widetilde{v}} \right)^{m_n}
\end{align*}
We now raise the result to the $1/m_n = \frac{2}{n(n+1)}$ power. The first two bracketed terms have the following limit:
\begin{align*}
  \lim_{n\rightarrow \infty}  \left[ \prod_{j=1}^n \frac{j}{(M+j)} \right]^{\frac{2 \mu}{\alpha n(n+1)}} &= 1 \\
  \lim_{n \rightarrow\infty} \left[ \frac{\tau_{\widetilde{v}}}{c}\right]^{\frac{2 M n}{n (n+1)} } &= 1
\end{align*}
The logarithm of the last bracketed term has the limit
\begin{align*}
  \lim_{n \rightarrow \infty}  \log \left[ \left[ \prod_{j=1}^n \left(\frac{j}{n}\right)^j \right]^{1/\alpha} \right]^{2/(n(n+1))} &= \frac{2}{\alpha} \lim_{n \rightarrow \infty} \frac{1}{n+1} \sum_{j=1}^n \left(\frac{j}{n}\right) \log \left(\frac{j}{n}\right) \\
  &= \frac{2}{\alpha} \int_0^1 x \log x \dx{x} \\
  &= -\frac{1}{2\alpha}
\end{align*}
Therefore, we have
\begin{align*}
  \lim_{n \rightarrow \infty} \left[ V(z_{0,n}, \ldots, z_{n,n}) \prod_{j=0}^n \widetilde{v}(z_{j,n}) \right]^{1/m_n} &=  \lim_{n \rightarrow \infty} \left[ \widetilde{v}^n(z_{0,n}) \prod_{j=1}^n \widetilde{v}^n(z_{j,n}) \prod_{k=0}^{j-1} \left| z_{j,n} - z_{k,n} \right|\right]^{1/m_n} \\
 &\geq \tau_{\widetilde{v}} \exp\left(-\frac{1}{2\alpha}\right) = \delta_{\widetilde{v}},
\end{align*}
and so we have proven \eqref{eq:hermite-wish}, that the weighted Leja points are asymptotically weighted Fekete.

For the case $\mu < 0$, we may repeat arguments for the Jacobi $\alpha,\beta<0$ case: $\mu < 0$ implies that node $z_1 = 0$ is chosen (assuming $z_0 \neq 0$), which then reverts the $w$-weighted Leja objective \eqref{eq:weighted-leja-objective} to one where $\mu > 0$.

\subsection{Laguerre case}\label{sec:laguerre-proof}
The Laguerre result for weighted Leja sequences defined by \eqref{eq:weighted-leja-objective} can easily be obtained by following the argument in Section \ref{sec:proof-hermite}, so we omit the details. We only mention that for weights of the form $\widetilde{v}(z) = \exp(- |z|)$ on $[0, \infty)$, we can directly obtain from \eqref{eq:chebyshev-transfinite-relation}:
\begin{align*}
  \tau_{\widetilde{v}} = \sqrt{e} \delta_{\widetilde{v}}.
\end{align*}
I.e., \eqref{eq:hermite-chebyshev-transfinite-relation} holds with $\alpha = 1$ \cite{mhaskar_extremal_1983}. With this, the remainder of the proof follows in precisely the same fashion as the Hermite case.

}{
  
}

\bibliographystyle{plain}
\bibliography{sparse-leja}

\end{document}